\documentclass[final,onefignum,onetabnum]{siamart190516}

\usepackage{amsfonts, mathtools, amssymb, amsopn, bm}
\usepackage{mathrsfs}
\usepackage{algorithmic}
\usepackage{booktabs, siunitx}
\usepackage{tikz, pgfplots, subfig, graphicx}
\pgfplotsset{compat=1.13}

\usepackage{todonotes}
\usepackage{epstopdf}

\DeclareMathOperator*{\argmin}{arg\,min}
\DeclareMathOperator*{\argmax}{arg\,max}
\DeclareMathOperator*{\supp}{supp}

\newcommand{\fun}[3]{#1 \colon #2 \rightarrow #3}
\newcommand{\abs}[1]{\left|#1\right|}
\newcommand{\norm}[2]{\left\| #1 \right\|_{#2}}

\renewcommand{\subset}{\subseteq}
\renewcommand{\epsilon}{\varepsilon}
\renewcommand{\b}{\bm}
\renewcommand{\cref}{\Cref}

\newcommand{\eps}{\mathrm\varepsilon}
\newcommand{\e}{\mathrm e}
\renewcommand{\i}{\mathrm i}

\newcommand{\C}{\mathbb C}
\newcommand{\N}{\mathbb N}
\newcommand{\R}{\mathbb R}
\newcommand{\Z}{\mathbb Z}
\newcommand{\T}{\mathbb T}
\renewcommand{\P}{\mathbb P}

\newcommand{\Td}{\T^d}

\newcommand{\x}{\b x}
\newcommand{\y}{\b y}
\renewcommand{\k}{\b k}
\newcommand{\h}{\b h}
\newcommand{\f}{\b f}
\newcommand{\F}{\b F}

\renewcommand{\L}{\mathrm{L}}
\newcommand{\Lt}{\L_2}
\newcommand{\sobolev}[1]{\mathrm{H}^{#1}}
\newcommand{\wiener}[1]{\mathcal{A}^{#1}}

\newcommand{\D}{\mathcal D}
\renewcommand{\u}{\b u}
\newcommand{\uc}{\b{u}^{\mathrm c}}
\newcommand{\au}{\abs{\u}}
\newcommand{\auc}{d-\abs{\u}}
\renewcommand{\v}{\bm v}
\newcommand{\vc}{\bm{v}^{\mathrm c}}
\newcommand{\av}{\abs{\v}}

\newcommand{\AT}[1]{\mathrm{T}_{#1}}
\newcommand{\Tds}{\AT{d_s}}

\newcommand{\Pset}[2]{\P_{#1}^{(#2)}}
\newcommand{\Pud}{\Pset{\u}{d}}

\newcommand{\Fset}[2]{\mathbb{F}_{#1}^{(#2)}}
\newcommand{\Fud}{\Fset{\u}{d}}

\newcommand{\va}[1]{\sigma^2(#1)} 
\newcommand{\gsi}[2]{\varrho(#1,#2)}

\newcommand{\Fseries}[4]{\sum_{#1 \in #2} #3 \, \e^{2\pi\i #1 \cdot #4}}

\newcommand{\fc}[2]{\mathrm{c}_{#1}\!\left(#2\right)}
\newcommand{\bfh}{\hat{\f}}

\newcommand{\dsuper}{\mathrm{d}^{(\mathrm{sp})}}
\newcommand{\scal}[2]{\bm #1 \cdot \bm #2}
\newcommand{\scall}[2]{#1 \cdot #2}
\newcommand{\Fser}{\Fseries}
\newcommand{\aw}{\wiener{w}}
\newcommand{\st}{\sobolev{w}}
\renewcommand{\d}{\mathrm d}
\newsiamremark{remark}{Remark}
\newsiamremark{hypothesis}{Hypothesis}
\crefname{hypothesis}{Hypothesis}{Hypotheses}
\newsiamthm{claim}{Claim}

\title{Approximation of high-dimensional periodic functions with Fourier-based methods}

\author{Daniel Potts\thanks{Chemnitz University of Technology, Germany 
		(\email{potts@math.tu-chemnitz.de}, \url{http://www.tu-chemnitz.de/\string~potts/}).}
	\and Michael Schmischke\thanks{Chemnitz University of Technology, Germany 
		(\email{michael.schmischke@math.tu-chemnitz.de}, \url{http://www.tu-chemnitz.de/\string~mischmi/}).}
}

\usepackage{framed}
\colorlet{shadecolor}{blue!20}

\headers{Approximation of high-dimensional periodic functions}{D. Potts, and M. Schmischke}

\ifpdf
\hypersetup{
  pdftitle={Approximation of high-dimensional periodic functions with Fourier-based methods},
  pdfauthor={D. Potts, and M. Schmischke}
}
\fi

\begin{document}

\maketitle

\begin{abstract}
	In this paper we propose an approximation method for high-dimensional $1$-periodic functions based on the multivariate ANOVA decomposition. We provide an analysis on the classical ANOVA decomposition on the torus and prove some important properties such as the inheritance of smoothness for Sobolev type spaces and the weighted Wiener algebra. We exploit special kinds of sparsity in the ANOVA decomposition with the aim to approximate a function in a scattered data or black-box approximation scenario. This method allows us to simultaneously achieve an importance ranking on dimensions and dimension interactions which is referred to as attribute ranking in some applications. In scattered data approximation we rely on a special algorithm based on the non-equispaced fast Fourier transform (or NFFT) for fast multiplication with arising Fourier matrices. For black-box approximation we choose the well-known rank-1 lattices as sampling schemes and show properties of the appearing special lattices.
\end{abstract}



\section{Introduction}

The approximation of high-dimensional functions is an important and current topic with great interest in many applications. We consider a setting of periodic functions $\fun{f}{\Td}{\C}$, $d \in \N$, over the torus $\T$ where certain data about the function is known. Here, we distinguish between a black-box setting, i.e., $f$ can be evaluated at points $\x \in \T^d$ at a certain cost, and a scattered data setting, i.e., sampling points $X \subset \T^d$ and function values $(f(\x))_{\x\in X}$ are given. Besides the natural question of wanting to find an approximation for $f$, we want to consider the question of interpretability, i.e., analyzing the importance of the dimensions and dimension interactions of the function. In applications this is sometimes referred to as an \textit{attribute ranking}.

The main tool to achieve our goals is the \textbf{analysis of variance (ANOVA)} decomposition \cite{CaMoOw97, RaAl99, LiOw06, Holtz11} which is an important model in the analysis of dimension interactions of multivariate, high-dimensional functions. It has proved useful in understanding the reason behind the success of certain quadrature methods for high-dimensional integration \cite{Ni92, BuGr04,GrHo10} and also infinite-dimensional integration \cite{BaGn14, GrKuSl16, KuNuPlSlWa17}. The ANOVA decomposition decomposes a $d$-variate function in $2^d$ ANOVA terms where each term belongs to a subset of $\D \coloneqq \{1,2,\dots,d\}$. The single term depends only on the variables in the corresponding subset and the number of these variables is the order of the ANOVA term. In this paper we study the classical ANOVA decomposition for periodic functions and how it acts on the frequency domain. The decomposition is referred to as \textit{classical} since it is based on an integral projection operator. In this setting we find relationships between ANOVA terms and the support of the frequencies as subsets of $\Z^d$. Moreover, we prove formulas for the representation of ANOVA terms and projections.     

Classical approximation methods cannot be applied for high-dimensional functions in general since the data required increases exponentially because of the curse of dimensionality. However, the observation has been made that in practical applications with multivariate, high-dimensional functions, often only the ANOVA terms of a low order are enough to describe a function, see e.g.~\cite{CaMoOw97}. This leads to the notion of a superposition dimension $d_s \in \D$ that limits the order of the ANOVA terms involved. Using this as a sparsity assumption to circumvent the curse of dimensionality, we consider functions where the ANOVA decomposition is mostly supported on terms of a low order, i.e., the norm of the remaining decomposition weighed by the norm of $f$ is small. This leads to a truncation of the decomposition through a superposition threshold. We consider how the previously described error can be related to the decay of Fourier coefficients and specifically the smoothness of $f$.

We present and analyze an approximation method that uses sensitivity analysis, cf.\ \cite{So90, So01, LiOw06}, on the truncated ANOVA decomposition which is able to identify important ANOVA terms and incorporate this information in finding an approximation. The goal is to simplify the approximation model which yields benefits in reducing the influence of overfitting regarding the amount of data. We determine approximations of the Fourier coefficients of the function (or \textit{learn} them) by solving least-squares problems. This is done through exploiting the special structure of the system matrix by identifying submatrices with the corresponding ANOVA terms. In the case of black-box approximations we are using rank-1 lattices as a spatial discretization, see e.g.~\cite{Kae2012,Kae2013,KaPoVo13,KaPoVo14}, and for scattered data approximation we rely on the iterative LSQR method \cite{PaSa82} and the fast matrix-vector multiplications for Fourier matrices provided by the non-equispaced fast Fourier transform (or NFFT) introduced in \cite{KeKuPo09}.

The paper is organized as follows. In \Cref{sec:anova} we introduce the classical ANOVA decomposition and study its behavior for periodic functions with regard to the Fourier system. We prove new formulas for the Fourier coefficients of projections in \cref{lemma:anova:projections} and ANOVA terms in \cref{lemma:anova:terms}. Moreover, we prove that functions in Sobolev type spaces and the weighted Wiener algebra inherit their smoothness to the ANOVA terms, see \cref{thm:inheritance1} and \cref{thm:inheritance2}. In \Cref{sec:trunc} we consider the truncated ANOVA decomposition and prove formulas for their Fourier coefficients, see \Cref{lemma:anova:trunc_coeff} and \cref{cor:trunc_coeff:ds}. We also give direct formulas for the truncated decomposition using the projections in \Cref{lemma:anova:direct_trunc} and \Cref{cor:trunc:direct_sup}. Furthermore, we relate Sobolev type spaces and the weighted Wiener algebra to the previously introduced functions of low-dimensional structure and compute the errors in \cref{theorem:anova:l2_error} and \cref{theorem:anova:inf_error_1}. Specifically, we consider a class of product and order-dependent weights, see \cite{KuSchwSl12, GrKuNi14, KuNu16, GrKuNu18}, of functions with isotropic and dominating-mixed smoothness, cf.\ \cite{GriHa13, KaPoVo13, ByKaUlVo16}, to obtain specific error bounds, see \cref{cor:bound:1} and \cref{cor:bound:2}. In \Cref{sec:anova:approxmethod} we present an approximation method for functions that are of an low-dimensional structure, cf.~\cref{alg}. We start by considering a black-box approximation scenario with rank-1 lattices as sampling schemes and show properties of the arising lattices in \cref{lem:approx:diff}, \cref{lem:approx:IUcard}, and \cref{cor:approx:diff_card}. Furthermore, we discuss scattered data approximation in \cref{sec:approx:scattered}. The arising approximation errors are considered in \cref{sec:anova:error} with main results being \Cref{thm:av_trunc_error:1}, \Cref{thm:av_trunc_error:2}, \cref{thm:anova:error:bb}, and \cref{thm:error:approx}. In \cref{sec:numeric} we perform numerical experiments with a specific benchmark function.

\section{Prerequisites and Notation}

We consider multivariate 1-periodic functions $\fun{f}{\T^d}{\C}$ with spatial dimension $d \in \N$, which are square-integrable, i.e., elements of $\Lt(\Td)$. Those functions have a unique representation with regard to the Fourier system $\{ \e^{2\pi\i\k\cdot\x} \}_{\bm k \in \Z^d}$ as Fourier series
\begin{equation*}
	f(\bm x) = \sum_{\bm k \in \Z^d} \fc{\k}{f} \, \e^{2\pi\i\scal{k}{x}},
\end{equation*}
where $\fc{\k}{f} \coloneqq \int_{\Td} f(\bm x)\,\e^{-2\pi\i\scal{k}{x}}\mathrm d\bm x \in \C$, $\bm k \in \Z^d$, are the Fourier coefficients of $f$. Given a finite index set $I \subset \Z^d$, we call the trigonometric polynomial
\begin{equation}\label{eq:pre:fps}
	S_I f (\x) = \sum_{\k\in I} \fc{\k}{f} \, \e^{2\pi\i\scal{k}{x}}
\end{equation}
the Fourier partial sum of $f$ with respect to the index set $I$.

In this paper we make use of indexing with sets. First, for a given spatial dimension $d$ we denote with $\D = \{ 1,2,\dots,d \}$ the set of coordinate indices and subsets as bold small letters, e.g., $\u \subset \D$. The complement of those subsets are always with respect to $\D$, i.e., $\uc = \D \setminus \u$. For a vector $\x \in \C^d$ we define $\x_{\u} = ( x_i )_{i \in \u} \in \C^{\au}$. There remains a small ambiguity regarding the order of the components of $\x_{\u}$ which can be clarified if one chooses a consistent ordering, e.g., ascending order which would be a natural choice.

Furthermore, we use the $p$-norm which is defined as
\begin{equation*}
	\norm{\x}{p} = \begin{cases}
		\abs{ \{ i \in \D \colon x_i \neq 0 \} } \quad&\colon p = 0 \\
		\left( \sum_{i=1}^d \abs{x_i}^p \right)^{1/p} &\colon 1 \leq p < \infty \\
		\max_{\i\in\D} \abs{x_i} &\colon p = \infty
	\end{cases}
\end{equation*}
for $\x\in\R^d$. Note that the case $1 \leq p < \infty$ can be expanded to $0 < p < 1$, but then $\norm{\cdot}{p}$ would only be a quasi-norm. In the case $p = 0$, $\norm{\cdot}{p}$ is not a norm at all.

\subsection{Rank-1 lattice}\label{sec:pre:lattice}

In the case of black-box approximation we are going to rely on rank-1 lattice as sampling schemes, see e.g.~\cite{Kae2012,Kae2013,KaPoVo13,KaPoVo14}. For a given lattice size $M \in \N$ and a generating vector $\b z \in \Z^d$ we define a rank-1 lattice \begin{equation*}
    \Lambda( \b z, M ) \coloneqq \left\{ \x_j \coloneqq \frac{j}{M}\b z \mod \b 1 \colon j = 0,1,\dots,M-1 \right\}.
\end{equation*} These lattices are useful in the evaluation of trigonometric polynomials $$p \in \Pi_I \coloneqq \mathrm{span} \left\{ \e^{2\pi\i\k\cdot\circ} \colon \k \in I \right\}$$ over a finite index set $I \subset \Z^d$ for given Fourier coefficients $\fc{\k}{p}$. As discussed in \cite{LiHi03}, we have \begin{equation*}
    p(\x_j) = p\left(\frac{j}{M}\b z \mod \b 1\right) = \sum_{l=0}^{M-1} \left( \sum_{\substack{\k\in I \\ \k\cdot\b z \equiv l \mod M}} \fc{\k}{p} \right) \e^{2\pi\i\frac{j l}{M}}.
\end{equation*} The computation of the sum over $l$ can be realized trough a one-dimensional FFT and therefore the evaluation of $p$ at all lattice nodes can be done using only this single FFT. The arithmetic complexity of this evaluation is in $\mathcal{O}(M\log M + d \abs{I})$.

However, using a special kind of rank-1 lattice, it is possible to reconstruct the Fourier coefficients $\fc{\k}{p}$ by sampling $p$ at the nodes in $\Lambda(\b z,M)$ in an exact and stable way. For an index set $I \subset \Z^d$ we define a reconstructing rank-1 lattice $\Lambda(\b z,M,I)$ as a rank-1 lattice $\Lambda(\b z,M)$ such that the condition \begin{equation*}
    \b m \cdot \b z \not\equiv 0 \mod M \quad\forall \b m \in \D(I)\setminus\{\b 0\}
\end{equation*} is fulfilled with \begin{equation}\label{eq:pre:di}
    \D(I) \coloneqq \left\{ \k - \b h \colon \k,\b h \in I \right\}
\end{equation} being the difference set for $I$. Using the nodes of a reconstructing rank-1 lattice $\Lambda(\b z,M,I)$, the Fourier coefficients can be calculated as \begin{equation*}
    \fc{\k}{p} = \frac{1}{M} \sum_{j=0}^M p\left(\frac{j}{M}\b z \mod \b 1\right) \e^{-2\pi\i j\frac{\k \cdot \b z}{M}}.
\end{equation*} The calculation of all Fourier coefficients $\fc{\k}{p}$, $\k \in I$, can then be realized trough a one-dimensional FFT and the computation of the products $\k \cdot \b z$. Consequently, the arithmetic complexity of this evaluation is again in $\mathcal{O}(M\log M + d \abs{I})$.

The principle of this reconstruction can be generalized to functions $f \in \aw(\T^d) \coloneqq \{ f \in \mathrm{L}_1(\Td) \colon \norm{f}{\aw(\T^d)} \coloneqq \sum_{\bm k \in \Z^d} w(\k) \abs{\fc{\k}{f}} < \infty \} $, $\fun{w}{\Z^d}{[1,\infty)}$, by taking the Fourier partial sum $S_I f$ for a suitable index set $I \subset \Z^d$ and treating the evaluations of $f$ as the evaluations of the trigonometric polynomial $S_I f$. Using the same idea as before, we get \begin{equation*}
    \fc{\k}{f} \approx \hat{f}_{\k} \coloneqq \frac{1}{M} \sum_{j=0}^M f\left(\frac{j}{M}\b z \mod \b 1\right) \e^{-2\pi\i j\frac{\k \cdot \b z}{M}}
\end{equation*} with a a reconstructing rank-1 lattice $\Lambda(\b z, M, I)$. The error for each coefficient is \begin{equation}\label{eq:pre:error}
    \hat{f}_{\k} = \fc{\k}{f} + \sum_{\h \in \Lambda^\bot(\b z, M)\setminus\{\b 0\}} \fc{\k + \h}{f}
\end{equation} with the integer dual lattice \begin{equation*}
    \Lambda^\bot(\b z, M) \coloneqq \left\{ \k \in \Z^d \colon \k\cdot\b z \equiv 0 \mod M \right\}.
\end{equation*} Consequently, if $\sum_{\h \in \Lambda^\bot(\b z, M)\setminus\{\b 0\}} \abs{\fc{\k + \h}{f}}$ is small, the approximations $\hat{f}_{\k}$ are close to the Fourier coefficients $\fc{\k}{f}$. For further details on this topic we refer to \cite{KaPoVo13,KaPoVo14} and \cite[Chapter 8]{PlPoStTa18}.

\section{The classical ANOVA decomposition of 1-periodic functions}\label{sec:anova}

In this section we introduce the ANOVA decomposition, see e.g.\ \cite{CaMoOw97, LiOw06, Holtz11}, and derive new results for the periodic setting specifically with regard to the decomposition acting on the frequency domain.

We start by defining the projection operator
\begin{equation}\label{eq:anova:projection_operator}
\mathrm{P}_{\u} f ( \bm{x}_{\u} ) \coloneqq \int_{\T^{\auc}} f(\bm x) \d\bm{x}_{\uc}
\end{equation}
that integrates over the variables $\bm{x}_{\uc}$. For $\au > 0$ this operator maps a function from $\Lt(\Td)$ to $\mathrm{L}_2(\T^{\au})$ by the Cauchy-Schwarz inequality and the image $\mathrm{P}_{\u} f$ depends only on the variables $\bm{x}_{\u} \in \T^{\au}$. In the case of $\u = \emptyset$, the projection gives us the integral of $f$. We define the index set
\begin{equation}\label{eq:anova:pud}
\Pud \coloneqq \left\{ \bm k \in \Z^d \colon \bm{k}_{\uc} = \bm 0 \right\}
\end{equation}
which can be identified with $\Z^{\au}$ using the mapping $\bm k \mapsto \bm{k}_{\u}$. Note that we use the convention $\Z^{\abs{\emptyset}} = \{ 0 \}$. We now prove a relationship between the Fourier coefficients of $\mathrm{P}_{\u} f$ and $f$.
\begin{lemma}\label{lemma:anova:projections}
	Let $f \in \Lt(\T^d)$ and $\bm\ell \in \Z^{\au}$. Then
	\begin{equation*}
	\fc{\b\ell}{\mathrm{P}_{\u} f} = \fc{\k}{f}
	\end{equation*}
	for $\bm k \in \Z^d$ with $\bm{k}_{\u} = \bm\ell$ and $\bm{k}_{\uc} = \bm 0$.
\end{lemma}
\begin{proof}
	We consolidate the two integrals and derive 
	\begin{align*}
	\fc{\b\ell}{\mathrm{P}_{\u} f} &= \int_{\T^{\au}} \int_{\T^{\auc}} f(\bm x) \d\bm{x}_{\uc}\,\e^{-2\pi\i\scall{\bm \ell}{\bm{x}_{\bm u}}}\d\bm{x}_{\u} \\
	&= \int_{\Td} f(\bm x)\,\e^{-2\pi\i\scall{\bm \ell}{\bm{x}_{\bm u}}} \d\bm{x} \\
	&= \int_{\Td} f(\bm x)\,\e^{-2\pi\i\scall{\bm k}{\bm{x}}} \d\bm{x} = \fc{\k}{f}.
	\end{align*}
\end{proof} Using \cref{lemma:anova:projections}, we are able to write $\mathrm{P}_{\u} f$ as both, a $d$-dimensional Fourier series $\mathrm{P}_{\u} f (\bm x) = \sum_{\bm k \in \Pud} \fc{\k}{f} \, \e^{2\pi\i\scal{k}{x}}$ and a $\au$-dimensional Fourier series $\mathrm{P}_{\u} f (\bm{x}_{\u}) = \sum_{\bm \ell \in \Z^{\au}} \fc{\ell}{\mathrm{P}_{\u} f} \, \e^{2\pi\i\scall{\bm\ell}{\bm{x}_{\u}}}$.

Now, we recursively define the \textbf{ANOVA term} for $\u \subseteq \D$
\begin{equation}\label{eq:anova:term}
f_{\u} \coloneqq \mathrm{P}_{\u} f - \sum_{\bm v \subsetneq \bm u } f_{\v}.
\end{equation} There exists a direct formula for the ANOVA terms $f_{\u}$ defined in \eqref{eq:anova:term}. 

\begin{lemma}\label{sup:lem:help}
	Let $a \in \N_{0}$ and $b \in \N$ with $b > a$. Then
	\begin{equation*}
		\sum_{n=a}^{b-1} (-1)^{n-a+1} \binom{b-a}{n-a} = (-1)^{b-a}.
	\end{equation*}
\end{lemma}

\begin{proof}
	We prove an equivalent form obtained through multiplication with $(-1)^a$ and an index shift
	\begin{equation*}
		\sum_{n=0}^{b-a-1} (-1)^{n+a+1} \binom{b-a}{n} = (-1)^{b}.
	\end{equation*}
	Splitting the sum and applying the Binomial theorem yields
	\begin{align*}
		\sum_{n=0}^{b-a-1} (-1)^{n+a+1} \binom{b-a}{n} &= \sum_{n=0}^{b-a} (-1)^{n+a+1} \binom{b-a}{n} - (-1)^{b+1} \\
		&= (-1)^{a+1} \underbrace{\sum_{n=0}^{b-a}(-1)^{n} \binom{b-a}{n}}_{=(-1+1)^{b-a}=0} + (-1)^{b} \\
		&= (-1)^b.
	\end{align*}
\end{proof}

\begin{lemma}\label{lemma:anova:direct}
	Let $f \in \Lt(\T^d)$ with $\u \subseteq \D$. Then
	\begin{equation}\label{eq:ANOVA:direct}
	f_{\u} = \sum_{\v \subset \u} (-1)^{\au-\av} \mathrm{P}_{\v} f.
	\end{equation}
\end{lemma}  
\begin{proof}
	A proof based on properties of projection operators was given in \cite{KuSlWaWo09} while we use combinatorial arguments. We prove this statement through structural induction over the cardinality of $\b u$. For $\abs{\b u} = 0$, i.e., $\b u = \emptyset$, we have
	\begin{equation*}
		(-1)^{0-0}\left(P_{\emptyset}\right)(\b x) = \left(P_{\emptyset}\right)(\b x) = \left(P_{\emptyset}\right)(\b x) - \sum_{\b v \subsetneq \emptyset} f_{\b v}(\b x).
	\end{equation*} 
	Now, let \eqref{eq:ANOVA:direct} be true for $\v \subset \D$, $\av = 0,1,\dots,m-1$, $m \in \{1,2,\dots,d\}$, and take a subset $\b u \subset \D$ with $\abs{\b u} = m$. We use the notation
	\begin{equation*}
		\delta_{\b w \subset \b v} = \begin{cases}
			1 \quad \colon\b w \subset \b v \\
			0 \quad \colon\text{otherwise.}
		\end{cases}
	\end{equation*}
	and start from the recursive expression in \eqref{eq:anova:term} to obtain
	\begin{align*}
		f_{\b u}(\b x) &= \left(P_{\b u}f\right)(\b x) - \sum_{\b v \subsetneq \b u} f_{\b v}(\b x) 
		= \left(P_{\b u}f\right)(\b x) - \sum_{\b v \subsetneq \b u} \sum_{\b w \subset \b v} (-1)^{\abs{\b v}-\abs{\b w}}\left(P_{\b w}f\right)(\b{x}) \\
		&= \left(P_{\b u}f\right)(\b x) - \sum_{\b v \subsetneq \b u} \sum_{\b w \subsetneq \b u} (-1)^{\abs{\b v}-\abs{\b w}}\left(P_{\b w}f\right)(\b{x})\delta_{\b w \subset \b v}.
	\end{align*}
	We exchange the two sums and sum over the order of the ANOVA terms 
	\begin{align*}
		\sum_{\b v \subsetneq \b u} \sum_{\b w \subsetneq \b u} (-1)^{\abs{\b v}-\abs{\b w}}\left(P_{\b w}f\right)(\b{x})\delta_{\b w \subset \b v} &= \sum_{\b w \subsetneq \b u} \left(P_{\b w}f\right)(\b{x})\sum_{\b v \subsetneq \b u} (-1)^{\abs{\b v}-\abs{\b w}} \delta_{\b w \subset \b v} \\
		&= \sum_{\b w \subsetneq \b u} \left(P_{\b w}f\right)(\b{x})\sum_{n=\abs{\b w}}^{m-1}\sum_{\substack{\b v \subset \b u \\ \abs{\b v} = n}} (-1)^{\abs{\b v}-\abs{\b w}} \delta_{\b w \subset \b v} \\
		&= \sum_{\b w \subsetneq \b u} \left(P_{\b w}f\right)(\b{x})\sum_{n=\abs{\b w}}^{m-1}(-1)^{n-\abs{\b w}}\sum_{\substack{\b v \subset \b u \\ \abs{\b v} = n}}  \delta_{\b w \subset \b v}.
	\end{align*}
	Applying \cref{sup:lem:help} yields the formula.
\end{proof}

We proceed to present a relationship between the Fourier coefficients of $f_{\u}$ and $f$. Furthermore, we prove $f_{\u} \in \Lt(\T^{\au})$. Therefore, we define the index set
\begin{equation*}
\Fud \coloneqq \left\{ \bm k \in \Z^d \colon \bm{k}_{\uc} = \bm 0, k_j \neq 0 \,\forall j \in \bm u \right\}
\end{equation*}
which can be identified with $( \Z \setminus \{ 0 \} )^{\au}$ using the mapping $\bm k \mapsto \bm{k}_{\u}$. Here, we use the convention $( \Z \setminus \{ 0 \} )^{\abs{\emptyset}} = \{ 0 \}$. 
\begin{lemma}\label{lemma:anova:sets}
	Let $\u,\v \subseteq \D$ with $\u \neq \v$. Then $\Fud \cap \mathbb{F}_{\v}^{(d)} = \emptyset$. Moreover, we have
	\begin{equation*}
	\Z^d = \bigcup_{\u \subseteq \D} \Fud.
	\end{equation*}
\end{lemma}
\begin{proof}
	Let $\u,\v \subseteq \D$, $\u \neq \v$, and w.l.o.g.\ $\abs{\u} \geq \abs{\v}$. We assume there exists a $\tilde{\bm k} \in \Fud \cap \mathbb{F}_{\v}^{(d)}$ and first consider the case $\u \cap \v = \emptyset$. Since $\tilde{\bm k} \in \Fud$ we have $\tilde{\bm k}_{\uc} = \bm 0$ and therefore $\tilde{\bm k}_{\v} = \bm 0$. This contradicts $\tilde{\bm k} \in \mathbb{F}_{\v}^{(d)}$. In the case of $\u \cap \v \neq \emptyset$ there exists a $j \in \vc \cap \u$. Then $\tilde{\bm k} \in \mathbb{F}_{\v}^{(d)}$ implies that $\tilde{k}_j = 0$ which contradicts $\tilde{\bm k} \in \Fud$.
	
	The inclusion $\bigcup_{\u \subseteq \D} \Fud \subseteq \Z^d$ is trivial since $\Fud \subseteq \Z^d$ for every $\u \subseteq \D$. To prove $\Z^d \subseteq \bigcup_{\u \subseteq \D} \Fud$ we take a $\bm k \in \Z^d$ and define $\u = \{ j \in \D \colon k_j \neq 0 \}$. Then $\bm k \in \Fud$ and therefore $\bm k \in \bigcup_{\u \subseteq \D} \Fud$.
\end{proof}

\begin{lemma}\label{lemma:anova:terms}
	Let $f \in \Lt(\T^d)$ with $\u \subseteq \D$ and $\bm\ell \in \Z^{\au}$. Then
	\begin{equation*}
	\fc{\ell}{f_{\u} } = \begin{cases}
	\fc{\k}{f} \quad &\colon \b\ell \in \left(\Z\setminus\{0\}\right)^{\au} \\
	\delta_{\u,\emptyset}\cdot\fc{0}{f} \quad &\colon \b\ell = \b 0 \\
	0 &\colon \text{otherwise} \\
	\end{cases}
	\end{equation*}
	for $\bm k \in \Z^d$ with $\bm{k}_{\u} = \bm\ell$ and $\bm{k}_{\uc} = \bm 0$. Furthermore, $f_{\u} \in \Lt(\T^{\au})$.
\end{lemma}
\begin{proof}
	We begin by employing the direct formula \eqref{eq:ANOVA:direct} to obtain
	\begin{align*}
	\fc{\ell}{f_{\u}} &= \int_{\T^{\au}} f_{\u}(\x_{\u}) \,\e^{-2\pi\i\scall{\b\ell}{\x_{\u}}} \d\x_{\u} \\
	&= \int_{\T^{\au}} \left[\sum_{\v \subset \u} (-1)^{\au-\av} \mathrm{P}_{\v} f (\x_{\v})\right] \e^{-2\pi\i\scall{\b\ell}{\x_{\u}}}  \d\x_{\u} \\
	&= \sum_{\v \subset \u} (-1)^{\au-\av} \int_{\T^{\au}} \mathrm{P}_{\v} f (\x_{\v})\, \e^{-2\pi\i\scall{\b\ell}{\x_{\u}}}  \d\x_{\u} \\
	&= \sum_{\v \subset \u} (-1)^{\au-\av} \mathrm{c}_{\k_{\v}}\left(\mathrm{P}_{\v}f\right) \delta_{\k_{\u\setminus\v},\b 0}.
	\end{align*}
	We go on to prove $\fc{0}{f_{\u}} = \delta_{\u,\emptyset}\cdot\fc{0}{f}$. In this case, $\k_{\v} = \b 0$ and $\delta_{\k_{\u\setminus\v},\b 0} = 1$ for every $\v \subset \u$. By the Binomial Theorem, we have
	\begin{align*}
	\fc{\ell}{f_{\u}} &= \sum_{\v \subset \u} (-1)^{\au-\av} \mathrm{c}_{\k_{\v}}\left(\mathrm{P}_{\v}f\right) \delta_{\k_{\u\setminus\v},\b 0} = \fc{0}{f} \sum_{\v \subset \u} (-1)^{\au-\av} \\
	&= \fc{0}{f} \sum_{n=0}^{\au} \binom{\au}{n} (-1)^{\au-n} = \fc{0}{f}\cdot\delta_{\u,\emptyset}.
	\end{align*}
	For the second case, we consider an $\b\ell$ and with a set $\overline{\v} \subset \u$ such that $\emptyset\neq\overline{\v} \coloneqq \{ i \in \u \colon k_i = 0 \}\neq\u$. Then $\delta_{\k_{\u\setminus\v},\b 0} = 1 \Longleftrightarrow \overline{\v}^\mathrm{c} \coloneqq \u\setminus\overline{\v} \subset \v$ and with the Binomial Theorem we get
	\begin{align*}
	\fc{\ell}{f_{\u}} &= \sum_{\v \subset \u} (-1)^{\au-\av} \mathrm{c}_{\k_{\v}}\left(\mathrm{P}_{\v}f\right) \delta_{\k_{\u\setminus\v},\b 0} = \sum_{\overline{\v}^\mathrm{c} \subset \v \subset \u} (-1)^{\au-\av} \mathrm{c}_{\k_{\v}}\left(\mathrm{P}_{\v}f\right) \\
	&= \fc{\k}{f} \sum_{\overline{\v}^\mathrm{c} \subset \v \subset \u} (-1)^{\au-\av} = \fc{\k}{f} \sum_{n=\abs{\overline{\v}^\mathrm{c}}}^{\au} \binom{\au-\abs{\overline{\v}^\mathrm{c}}}{n-\abs{\overline{\v}^\mathrm{c}}} (-1)^{\au-n} \\
	&= \fc{\k}{f} \sum_{m=0}^{\au-\abs{\overline{\v}^\mathrm{c}}} \binom{\au-\abs{\overline{\v}^\mathrm{c}}}{m} (-1)^{\au-\abs{\overline{\v}^\mathrm{c}}-m} = 0.
	\end{align*}
	For the case were the entries of $\b\ell$ are all nonzero, only the addend where $\v = \u$ is nonzero, i.e., $\fc{\ell}{f_{\u}} = \fc{\k}{f}$.
\end{proof}

With \cref{lemma:anova:terms} we have two equivalent series representations for the ANOVA term $f_{\u}$, the $d$-dimensional Fourier series $f_{\u} (\bm x) = \sum_{\bm k \in \Fud} \fc{\k}{f} \, \e^{2\pi\i\scal{k}{x}}$ and the $\au$-dimensional Fourier series $f_{\u} (\bm{x}_{\u}) = \sum_{\bm \ell \in \Z^{\au}} \fc{\ell}{f_{\u}} \, \e^{2\pi\i\scall{\bm\ell}{\bm{x}_{\u}}}$ with $\fc{\ell}{f_{\u}}$ as in \cref{lemma:anova:terms}. The ANOVA terms have the following important property.
\begin{corollary}\label{cor:ortho}
	Let $f \in \Lt(\T^d)$ and $\u,\v \subseteq \D$ with $\u \neq \v$. Then the ANOVA terms $f_{\u}$ and $f_{\v}$ are orthogonal, i.e.,
	\begin{equation*}
	\langle f_{\u}, f_{\v} \rangle_{\Lt(\T^d)} = 0.
	\end{equation*}
\end{corollary}
\begin{proof}
	We employ \cref{lemma:anova:sets} and \cref{lemma:anova:terms} to deduce
	\begin{align*}
		\langle f_{\u}, f_{\v} \rangle_{\Lt(\T^d)} &= \langle \sum_{\bm k \in \Fud} \fc{\k}{f} \, \e^{2\pi\i\scal{k}{x}}, \sum_{\bm \ell \in \Fud} \fc{\ell}{f} \, \e^{2\pi\i\scal{\ell}{x}} \rangle_{\Lt(\T^d)} \\
		&=  \sum_{\bm k \in \Fud} \sum_{\bm \ell \in \Fud} \fc{\k}{f} \overline{\fc{\ell}{f}}\,\delta_{\k,\bm\ell} = 0.
	\end{align*}
\end{proof}

Having defined the ANOVA terms, we now go on to the \textbf{ANOVA decomposition}, cf.\ \cite{CaMoOw97,LiOw06}.
\begin{theorem}\label{thm:anova_decomp}
	Let $f \in \Lt(\T^d)$, the ANOVA terms $f_{\u}$ as in \eqref{eq:anova:term} and the set of coordinate indices $\D = \{1,2,\dots,d\}$. Then f can be uniquely decomposed as
	\begin{equation}\label{eq:ANOVA:decomp}
	f(\b x) = f_\emptyset + \sum_{i=1}^{d} f_{ \{i\} } (x_i) + \sum_{i = 1}^{d-1} \sum_{j = i+1}^{d} f_{ \{i,j\} } (\b{x}_{\{i,j\}})+ \dots + f_\D(\x) = \sum_{\u \subseteq \D} f_{\b u}(\b{x}_{\u}) 
	\end{equation}
	which we call \textbf{analysis of variance (ANOVA) decomposition}.
\end{theorem}
\begin{proof}
	We use that $\Z^d$ is the disjoint union of the sets $\Fud$ for $\u \subset \D$ and obtain
	\begin{align*}
		\sum_{\u \subset \D} f_{\u}( \x_{\u} ) &= \sum_{\u \subset \D} \sum_{\bm k \in \Fud} \fc{k}{f} \, \e^{2\pi\i\scal{k}{x}} = \sum_{\bm k \in \bigcup_{\u \subseteq \D} \Fud} \fc{k}{f} \, \e^{2\pi\i\scal{k}{x}} \\
		&= \sum_{\bm k \in \Z^d} \fc{k}{f} \, \e^{2\pi\i\scal{k}{x}} = f(\b x).
	\end{align*}
	Since the union is disjoint, the decomposition is unique.
\end{proof}
\begin{remark}
	The ANOVA decomposition \eqref{eq:ANOVA:decomp} depends strongly on the projection operator $\mathrm{P}_{\u} f$, see \eqref{eq:anova:projection_operator}. The integral operator considered in this paper leads to the so called classical ANOVA decomposition. Another important variant is the anchored decomposition where one chooses an anchor point $\b c \in \T^d$ and the projection operator is then defined as 
	\begin{equation*}
	\mathrm{P}_{\u} f (\x_{\u}) = f(\b y), \,\b{y}_{\u} = \x_{\u}, \b{y}_{\uc} = \b{c}_{\uc}.
	\end{equation*}
	This decomposition can for example be used in methods for the integration of high-dimensional functions such as the multivariate decomposition method, see e.g.\ \cite{KuNuPlSlWa17, GiKuNuWa18}. However, the error analysis may again be based on the classical ANOVA decomposition, see e.g.\ \cite{GnHeHiRi17}. 
\end{remark}

In \Cref{fig:ANOVA:hypercube} we have visualized the different frequency index sets $\Fud$, $\u \subset \D$, for a $3$-dimensional example.

\begin{figure}[ht]
	\centering
	\subfloat[$\mathbb{F}_{\emptyset}^{(3)}$, $\mathbb{F}_{\{1\}}^{(3)} \cap {[-8,8]}^3$, $\mathbb{F}_{\{2\}}^{(3)} \cap {[-8,8]}^3$, and $\mathbb{F}_{\{3\}}^{(3)} \cap {[-8,8]}^3$]{
		\includegraphics{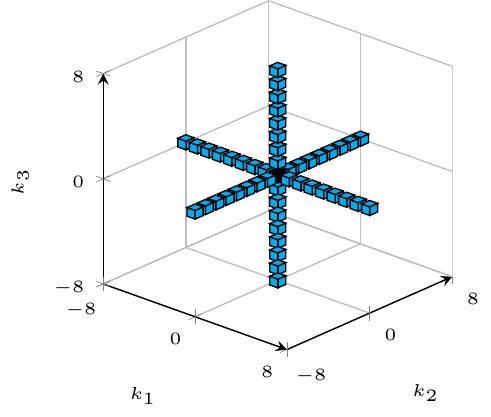}
	}
	\hfill
	\subfloat[$\mathbb{F}_{\{1,2\}}^{(3)} \cap {[-8,8]}^3$]{
		\includegraphics{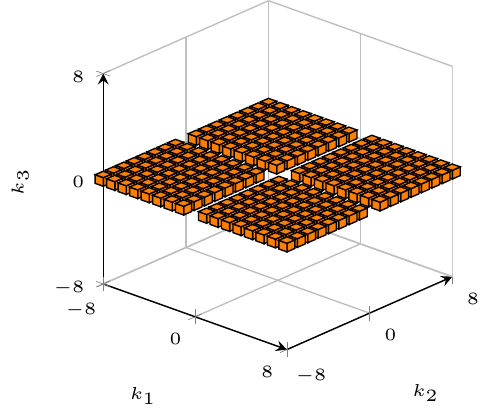}
	}
	\\
	\subfloat[$\mathbb{F}_{\{1,3\}}^{(3)} \cap {[-8,8]}^3$]{
		\includegraphics{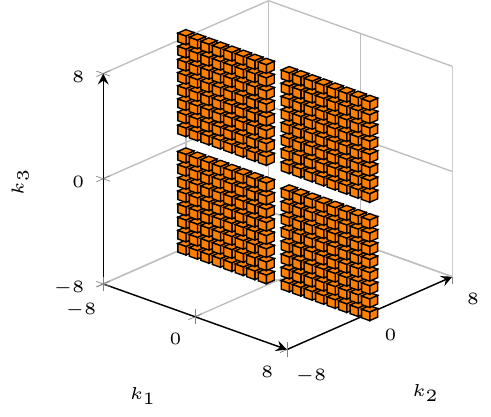}
	}
	\hfill
	\subfloat[$\mathbb{F}_{\{2,3\}}^{(3)} \cap {[-8,8]}^3$]{
		\includegraphics{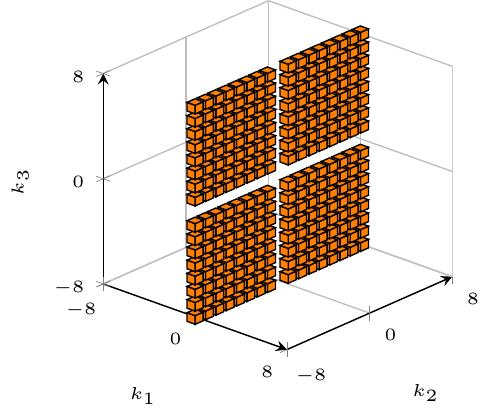}
	}
	\\
	\subfloat[$\mathbb{F}_{\{1,2,3\}}^{(3)} \cap {[-8,8]}^3$]{
		\includegraphics{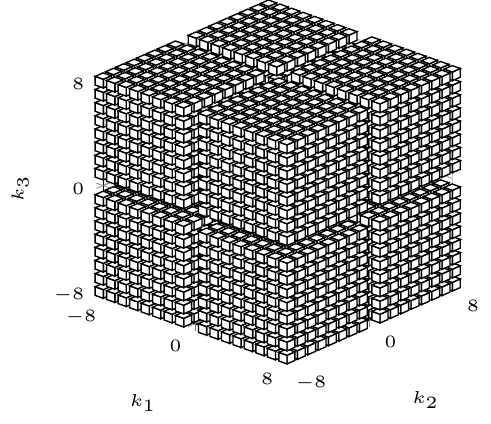}
	}
	\caption{The ANOVA decomposition working on the hypercube $[-8,8]^3$ as a part of the 3-dimensional index set $\Z^3$.}
	\label{fig:ANOVA:hypercube}
\end{figure}

\subsection{Variance and Sensitivity}

In order to get a notion of the importance of single terms compared to the entire function, we define the \textbf{variance of a function}
\begin{equation*}
\va{f} \coloneqq \int_{\T^d} \left( f(\x) - \fc{\b 0}{f} \right)^2 \,\d\x
\end{equation*}
for real-valued $f$. In this case, we have the equivalent formulation
\begin{equation*}
\va{f} = \norm{f}{\Lt(\T^d)}^2 - \abs{\fc{\b 0}{f}}^2
\end{equation*}
which yields a sensible definition for complex-valued functions $f$. For the ANOVA terms $f_{\u}$ with $\emptyset \neq \u \subset \D$ we have $\fc{\b 0}{f_{\u}} = 0$ and therefore
\begin{equation*}
\va{f_{\u}} = \norm{f_{\u}}{\Lt(\T^{\au})}^2.
\end{equation*}
\begin{lemma}\label{lemma:anova:variance}
	Let $f \in \Lt(\T^d)$. Then we obtain for the variance
	\begin{equation*}
	\va{f} = \sum_{\emptyset\neq\u\subset\D} \va{f_{\u}}.
	\end{equation*}
\end{lemma}
\begin{proof}
	We show that the right-hand side equals the left-hand side by employing \cref{lemma:anova:sets} and \cref{lemma:anova:terms} 
	\begin{align*}
	\sum_{\emptyset\neq\u\subset\D} \va{f_{\u}} &= \sum_{\emptyset\neq\u\subset\D} \sum_{\k\in\Fud} \abs{\fc{\k}{f}}^2 = \sum_{\k\in\bigcup_{\emptyset\neq\u \subseteq \D}\Fud} \abs{\fc{\k}{f}}^2 \\
	&= \sum_{\k\in\Z^d} \abs{\fc{\k}{f}}^2 - \abs{\fc{\b 0}{f}}^2 = \norm{f}{\Lt(\T^d)}^2 - \abs{\fc{\b 0}{f}}^2.
	\end{align*}
\end{proof}

The global sensitivity indices 
\begin{equation}\label{eq:anova:gsi}
\gsi{\u}{f} \coloneqq \frac{\va{f_{\u}}}{\va{f}} \in [0,1]
\end{equation}
for $\emptyset\neq\u\subset\D$ provide a comparable score to rank the importance of ANOVA terms against each other, cf.\ \cite{So90, So01, LiOw06}. Clearly, we have $\sum_{\emptyset\neq\u\subset\D} \gsi{\u}{f} = 1$ by \cref{lemma:anova:variance}.

We now introduce one notion of \textbf{effective dimensions} as proposed in \cite{CaMoOw97}. Given a fixed $\delta \in (0,1]$, the general notion of \textbf{superposition dimension} is defined as the minimum \begin{equation*}
\min  \left\{ s \in \D \colon \sum_{\substack{\emptyset\neq\u \subset \D \\ \au \leq s}} \va{f_{\u}} \geq \delta \va{f} \right\}. 
\end{equation*} If we consider a particular Hilbert space $H \subset \L_2(\T^d)$ with norm $\norm{\cdot}{H}$, we modify the superposition dimension in the sense of this space, see e.g.~\cite{Owen2019}. For $f \in H$ and $\delta \in (0,1]$ we define the modified superposition dimension as \begin{equation}\label{eq:anova:superposition_specific}
	\dsuper \coloneqq \min \left\{ s \in \D \colon \sup_{\norm{f}{H} \leq 1} \sum_{\au > s} \norm{f_{\u}}{\L_2(\T^d)}^2 \leq 1 - \delta \right\}.
\end{equation}

Finally, we investigate how the smoothness of $f$ translates to projections $\mathrm{P}_{\u} f$ and ANOVA terms $f_{\u}$. For a different setting this has been discussed in \cite{LiOw06, GrKuSl10, GrKuSl16} and therein called \textit{inheritance of smoothness}. In our setting, we express smoothness through special subspaces of $\Lt(\Td)$ and how $f$ being an element of those spaces translates to the projections $\mathrm{P}_{\u} f$ and ANOVA terms $f_{\u}$. In particular, we look at \textbf{Sobolev type spaces}, cf.\ \cite{KuMaUl16}, 
\begin{equation*}
\st(\T^d) \coloneqq \left\{ f \in \Lt(\Td) \colon \norm{f}{\st(\T^d)} \coloneqq \left(\sum_{\bm k \in \Z^d} w^2(\k) \abs{\fc{\k}{f}}^2\right)^{\frac{1}{2}} < \infty \right\} 
\end{equation*}
and the \textbf{weighted Wiener algebra}
\begin{equation*}
\aw(\T^d) \coloneqq \left\{ f \in \mathrm{L}_1(\Td) \colon \norm{f}{\aw(\T^d)} \coloneqq \sum_{\bm k \in \Z^d} w(\k) \abs{\fc{\k}{f}} < \infty \right\} 
\end{equation*}
with a weight function $\fun{w}{\Z^d}{[1,\infty)}$ for both cases. 

\begin{theorem}[Inheritance of smoothness for Sobolev type spaces]\label{thm:inheritance1}
	Let $f \in \st(\T^d)$ with weight function $\fun{w}{\Z^d}{[1,\infty)}$. Then for any weight function \\ $\fun{w_{\u}}{\Z^{\au}}{[1,\infty)}$ with
	\begin{equation*}
	w_{\u}(\bm{k}_{\u}) \leq w(\k) \,\forall \k \in \Pud
	\end{equation*}
	we have $\mathrm{P}_{\u} f \in \mathrm{H}^{w_{\u}}(\T^{\au})$ and $f_{\u} \in \mathrm{H}^{w_{\u}}(\T^{\au})$.
\end{theorem}
\begin{proof}
	We show that the norm $\norm{\mathrm{P}_{\u} f}{\mathrm{H}^{w_{\u}}(\T^{\au})}$ is finite by using \cref{lemma:anova:projections}
	\begin{align*}
	\sum_{\bm \ell \in \Z^{\au}} w_{\u}^2(\b\ell) \abs{\fc{\ell}{\mathrm{P}_{\u} f}}^2 &= \sum_{\k \in \Pud} w_{\u}^2(\k_{\u}) \abs{\fc{\k}{f}}^2 \leq \sum_{\k \in \Pud} w^2(\k) \abs{\fc{\k}{f}}^2 \\
	&\leq \sum_{\k \in \Z^d} w^2(\k) \abs{\fc{\k}{f}}^2 = \norm{f}{\st(\T^d)}^2 < \infty.
	\end{align*} 
	Analogously, we employ \cref{lemma:anova:terms} to prove $f_{\u} \in \mathrm{H}^{w_{\u}}(\T^{\au})$
	\begin{align*}
	\sum_{\bm \ell \in \Z^{\au}} w_{\u}^2(\b\ell) \abs{\fc{\ell}{f_{\u}}}^2 &= \sum_{\k \in \Fud} w_{\u}^2(\k_{\u}) \abs{\fc{\k}{f}}^2 \leq \sum_{\k \in \Fud} w^2(\k) \abs{\fc{\k}{f}}^2 \\
	&\leq \sum_{\k \in \Z^d} w^2(\k) \abs{\fc{\k}{f}}^2 = \norm{f}{\st(\T^d)}^2 < \infty.
	\end{align*}
\end{proof}

\begin{theorem}[Inheritance of smoothness for the weighted Wiener algebra]\label{thm:inheritance2}
	Let $f \in \aw(\T^d)$ with weight function $\fun{w}{\Z^d}{[1,\infty)}$. Then for any weight function $\fun{w_{\u}}{\Z^{\au}}{[1,\infty)}$ with
	\begin{equation*}
	w_{\u}(\bm{k}_{\u}) \leq w(\k) \,\forall \k \in \Pud
	\end{equation*}
	we have $\mathrm{P}_{\u} f \in \wiener{w_{\u}}(\T^{\au})$ and $f_{\u} \in \wiener{w_{\u}}(\T^{\au})$.
\end{theorem}
\begin{proof}
	We use \cref{lemma:anova:projections} to show that $\mathrm{P}_{\u} f \in \aw(\T^{\au})$
	\begin{align*}
		\sum_{\bm \ell \in \Z^{\au}} w_{\u}(\b\ell) \abs{\fc{\ell}{f_{\u}}} &= \sum_{\k \in \Pud} w_{\u}(\k_{\u}) \abs{\fc{k}{f}} \leq \sum_{\k \in \Pud} w(\k) \abs{\fc{k}{f}}\\
		&\leq \sum_{\k \in \Z^d} w(\k) \abs{\fc{k}{f}} = \norm{f}{\aw(\T^d)} < \infty.
	\end{align*} 
	We utilize \cref{lemma:anova:terms} to prove $f_{\u} \in \wiener{w_{\u}}(\T^{\au})$
	\begin{align*}
		\sum_{\bm \ell \in \Z^{\au}} w_{\u}(\b\ell) \abs{\fc{\ell}{f_{\u}}} &= \sum_{\k \in \Fud} w_{\u}(\k_{\u}) \abs{\fc{k}{f}} \leq \sum_{\k \in \Fud} w(\k) \abs{\fc{k}{f}} \\
		&\leq \sum_{\k \in \Z^d} w(\k) \abs{\fc{k}{f}} = \norm{f}{\aw(\T^d)} < \infty.
	\end{align*}
\end{proof}

The inheritance of smoothness has special significance with regard to the numerical realization of the method presented in \Cref{sec:anova:approxmethod}. It ensures that the ANOVA terms $f_{\u}$ are at least as smooth as the function $f$ in consideration which is relevant for the quality of the approximation produced by the method.

\section{Truncated ANOVA decomposition}\label{sec:trunc}

The number of ANOVA terms of a function is equal to the cardinality of $\mathcal{P}(\D)=2^d$ and therefore grows exponentially in the dimension. This reflects the curse of dimensionality in a certain way and poses a problem for the approximation of a function. In this section we consider truncating the ANOVA decomposition, i.e., removing certain terms $f_{\u}$, and therefore creating a certain form of sparsity. We define a \textbf{subset of ANOVA terms} as a subset of the power set of $\D$, i.e., $U \subset \mathcal{P}(D)$, such that the inclusion condition
\begin{equation}\label{eq:trunc_inclusion}
\u \in U \Longrightarrow \forall \v\subset \u \colon \v \in U
\end{equation}
holds, cf.\ \cite[Chapter 3.2]{Holtz11}. This is necessary due to the recursive definition of the ANOVA terms, see \eqref{eq:anova:term}. 

For any subset of ANOVA terms $U$ we then define the truncated ANOVA decomposition as 
\begin{equation*}
\mathrm{T}_{U} f \coloneqq \sum_{\u\in U} f_{\u}.
\end{equation*}
A specific truncation idea can be obtained by relating to the superposition dimension $\dsuper$, see \eqref{eq:anova:superposition_specific}. For a chosen superposition threshold $d_s \in \D$ (that may or may not be equal to the superposition dimension $\dsuper$), we define $U_{d_s} \coloneqq \{ \u \subset \D \colon \abs{\u} \leq d_s \}$ and $\mathrm{T}_{d_s} \coloneqq \mathrm{T}_{U_{d_s}}$. We subsequently prove properties of both $\mathrm{T}_{U}$ in general and $\mathrm{T}_{d_s}$ in particular.
\begin{lemma}\label{lemma:anova:trunc_coeff}
	Let $f \in \Lt(\T^d)$ and $U \subset \mathcal{P}(D)$ be a subset of ANOVA terms. Then $\mathrm{T}_U f\in \Lt(\T^d)$ and for $\k \in \Z^d$ the Fourier coefficient is 
	\begin{equation*}
	\fc{\k}{\mathrm{T}_U f} = \begin{cases}
	\fc{\k}{f} \,\,\,&\colon \exists \u \in U \colon \k \in \Fud \\
	0 &\colon \text{otherwise}.
	\end{cases}
	\end{equation*}
\end{lemma}\begin{proof}
	Clearly, we have $\mathrm{T}_U f\in \Lt(\T^d)$. Let now $\k \in \Z^d$. Then there exists a $\u_0 \subset \D$ such that $\k \in \mathbb{F}_{\u_0}^{(d)}$. We employ \cref{lemma:anova:terms} and obtain
	\begin{align*}
	\fc{\k}{\mathrm{T}_U f} &= \int_{\T^d} \left( \sum_{\u\in U} f_{\u}(\x_{\u}) \right) \e^{2\pi\i\k\cdot\x} \d\x = \sum_{\u\in U} \int_{\T^d} f_{\u}(\x_{\u}) \e^{2\pi\i\k\cdot\x} \d\x \\
	&= \sum_{\u\in U} \int_{\T^{\au}} f_{\u}(\x_{\u}) \e^{2\pi\i\k_{\u}\cdot\x_{\u}} \d\x_{\u}\,\delta_{\k_{\u},\b 0} = \sum_{\u\in U} \fc{\k}{f} \,\delta_{\k_{\uc},\b 0}\,(1-\delta_{\k_{\u},\b 0}) \\
	&= \begin{cases}
	\fc{\k}{f} \,\,\,&\colon \u_0 \in U \\
	0 &\colon \text{otherwise}.
	\end{cases}
	\end{align*}
\end{proof}
\begin{corollary}\label{cor:trunc_coeff:ds}
	Let $f \in \Lt(\T^d)$ and $d_s \in \D$ a superposition threshold. Then $\mathrm{T}_{d_s}f\in \Lt(\T^d)$ and only the Fourier coefficients corresponding to $d_s$-sparse frequencies are nonzero, i.e.,  
	\begin{equation*}
	\fc{\k}{\mathrm{T}_{d_s}f} = \begin{cases}
	\fc{\k}{f} \,\,\,&\colon \left\| \k \right\|_0 \leq d_s \\
	0 &\colon \text{otherwise}.
	\end{cases}
	\end{equation*}
\end{corollary}
\begin{proof}
	Since $U_{d_s}$ is a subset of ANOVA terms, $\mathrm{T}_{d_s}f\in \Lt(\T^d)$ follows directly from \cref{lemma:anova:trunc_coeff}. Moreover, $\exists \u \in U_{d_s} \colon \k \in \Fud \Longleftrightarrow \left\| \k \right\|_0 \leq d_s$.
\end{proof}

The following lemma shows that the number of terms in $U_{d_s}$ is polynomial in $d$ for a fixed $d_s$ and therefore allows us to circumvent the curse of dimensionality in terms of the number of sets. \begin{lemma}\label{lem:polycard}
	We estimate the cardinality of $\abs{U_{d_s}}$ as follows
	\begin{equation*}
	\abs{U_{d_s}} < \left(\frac{\e d}{d_s}\right)^{d_s},
	\end{equation*}
	i.e., the number of terms in $U_{d_s}$ has polynomial growth in $d$ for fixed $d_s \in \D\setminus\{d\}$.
\end{lemma}
\begin{proof}
    We estimate the sum as follows
    \begin{align*}
    	\abs{U_{d_s}} = \sum_{n=0}^{d_s} \binom{d}{n} \leq \sum_{n=0}^{d_s} \frac{d^n d_s^n}{n!\, d_s^n} = \sum_{n=0}^{d_s} \left(\frac{d}{d_s}\right)^n \frac{d_s^n}{n!} \leq \left(\frac{d}{d_s}\right)^{d_s} \sum_{n=0}^{d_s} \frac{d_s^n}{n!}.
    \end{align*}
    Estimating the sum by the Taylor series for $\e^{d_s}$ yields the statement.
\end{proof}

In the following we show direct formulas for the truncated ANOVA decomposition based on the projections similarly as for the ANOVA terms, see \eqref{eq:ANOVA:direct}. \begin{lemma}\label{lemma:anova:direct_trunc}
	Let $f \in \Lt(\T^d)$ and $U \subset \mathcal{P}(D)$ a subset of ANOVA terms. Then we have the direct formula
	\begin{equation*}
	\mathrm{T}_U f = \sum_{\u \in U} \sum_{\substack{\v \in U \\ \u \subset \v}} (-1)^{\av-\au} \mathrm{P}_{\u} f.
	\end{equation*}
\end{lemma}
\begin{proof}
	We apply equation \eqref{eq:ANOVA:direct} and obtain immediately
	\begin{align*}
	\mathrm{T}_U f &= \sum_{\u \in U} f_{\u} = \sum_{\u \in U} \sum_{\v \subset \u} (-1)^{\au-\av} \mathrm{P}_{\v} f = \sum_{\u \in U} \sum_{\v \in U} (-1)^{\au-\av} \mathrm{P}_{\v} f\,\delta_{\v\subset\u} \\
	&= \sum_{\v \in U} \sum_{\substack{\u \in U \\ \v \subset \u}} (-1)^{\au-\av}\mathrm{P}_{\v}f.
	\end{align*}
\end{proof}
\begin{corollary}\label{cor:trunc:direct_sup}
	Let $f \in \Lt(\T^d)$ and $d_s \in \D$ a superposition threshold. Then we have the direct formula
	\begin{equation*}
	\mathrm{T}_{d_s} f = \sum_{\substack{\u \subset \D \\ \au \leq d_s}} \left[ \sum_{n=\au}^{d_s} (-1)^{n-\au} \binom{d-\au}{n-\au}  \right] \mathrm{P}_{\u} f.
	\end{equation*}
\end{corollary}
\begin{proof}
	Since the equality
	\begin{equation*}
	\sum_{\substack{\v \in U_{d_s} \\ \u \subset \v}} (-1)^{\av-\au} = \sum_{n=\au}^{d_s} (-1)^{n-\au} \binom{d-\au}{n-\au},
	\end{equation*}
	holds, we employ \cref{lemma:anova:direct_trunc} and the formula is proven.	
\end{proof}

The truncated ANOVA decomposition plays a major role in our approximation approach presented in \Cref{sec:anova:approxmethod}. Therefore we are interested in functions that can be approximated well by a truncated ANOVA decomposition. Specifically, we are looking to characterize functions such that the truncation operation by $\mathrm{T}_U f$ for different sets $U$ retains most of the function, i.e., we have a relative error \begin{equation}\label{eq:trunc_err}
    \frac{\norm{f-\mathrm{T}_U f}{H_1}}{\norm{f}{H_2}} < \varepsilon 
\end{equation}
with $\varepsilon > 0$, and $H_1, H_2$ certain subspaces of $\Lt(\T^d)$. It is especially interesting to characterize these functions by properties like the smoothness. To this end, we start by proving general bounds for Sobolev type spaces $\st(\T^d)$ and the weighted Wiener algebra $\aw(\T^d)$ to later relate this to weight functions $w$ defined by specific kinds of smoothness.

Moreover, this can be related to the superposition dimension $\dsuper$ for a $\delta \in (0,1]$, see \eqref{eq:anova:superposition_specific}. Let $H_1 = \L_2(\T^d)$ and $H_2 \in \{ \sobolev{w}(\T^d), \wiener{w}(\T^d)\}$ for a weight function $w$. If we choose truncation by a superposition threshold $d_s \in \D$ then the bound on the right-hand side $\varepsilon(d_s) \in (0,1)$ depends on $d_s$. Moreover, we have \begin{equation}\label{eq:specificds}
	\sup_{f \neq 0} \frac{\norm{f-\mathrm{T}_{d_s} f}{\L_2(\T^d)}^2}{\norm{f}{H_2}^2} = \sup_{\norm{f}{H_2} \leq 1} \sum_{\au > d_s} \norm{f_{\u}}{\L_2(\T^d)}^2 < \varepsilon(d_s) 
\end{equation} which follows from $\norm{f-\mathrm{T}_{d_s} f}{\L_2(\T^d)}^2 = \sum_{\au > d_s} \norm{f_{\u}}{\L_2(\T^d)}^2$. The modified superposition dimension $\dsuper$ will now be smaller or equal to $\min \{ d_s \in \D \colon \varepsilon(d_s) \leq 1 - \delta \}$, i.e., truncation by this minimum as superposition threshold is guaranteed to be effective in relation to $\delta$.

\begin{theorem}\label{theorem:anova:l2_error}
	Let $f \in \st(\T^d)$ with weight function $\fun{w}{\Z^d}{[1,\infty)}$. Then
	\begin{equation*}
	\frac{\norm{f-\mathrm{T}_U f}{\Lt(\T^d)}}{\norm{f}{\st(\T^d)}} \leq \frac{1}{\min_{\k\in\bigcup_{\substack{\u \subseteq \D \\ \u \notin U}}\Fud} w(\k)} .
	\end{equation*}
\end{theorem}
\begin{proof}
	We employ Parseval's identity and \cref{lemma:anova:trunc_coeff} to derive
	\begin{align*}
	\norm{f-\mathrm{T}_U f}{\Lt(\T^d)}^2 &= \sum_{\k\in\Z^d} \abs{\fc{\k}{f} - \fc{k}{\mathrm{T}_U f} }^2 = \sum_{\k\in\bigcup_{\substack{\u \subseteq \D \\ \u \notin U}}\Fud} \abs{\fc{\k}{f} }^2 \\
	&=  \sum_{\k\in\bigcup_{\substack{\u \subseteq \D \\ \u \notin U}}\Fud} \frac{w^2(\k)}{w^2(\k)}\abs{\fc{\k}{f} }^2 \\&\leq \frac{1}{\min_{\k\in\bigcup_{\substack{\u \subseteq \D \\ \u \notin U}}\Fud} w^2(\k)} \norm{f}{\st(\T^d)}^2.
	\end{align*}
\end{proof}

\begin{theorem}\label{theorem:anova:inf_error_1}
	Let $f \in \aw(\T^d)$ with weight function $\fun{w}{\Z^d}{[1,\infty)}$. Then
	\begin{equation*}
	\frac{\norm{f-\mathrm{T}_U f}{\mathrm{L}_\infty(\T^d)}}{\norm{f}{\aw(\T^d)}} \leq \frac{1}{\min_{\k\in\bigcup_{\substack{\u \subseteq \D \\ \u \notin U}}\Fud} w(\k)}.
	\end{equation*}
	For $f \in \st(\T^d)$ with a weight function $\fun{w}{\Z^d}{[1,\infty)}$ such that $\{1/w(\k)\}_{\k\in\Z^d} \in \ell_2$ we have
	\begin{equation*}
		\norm{f-\mathrm{T}_U f}{\mathrm{L}_\infty(\T^d)} \leq \sqrt{\sum_{\k\in\bigcup_{\substack{\u \subseteq \D \\ \u \notin U}}\Fud} \frac{1}{w^2(\k)}} \,\,\norm{f}{\st(\T^d)}.
	\end{equation*}
\end{theorem}
\begin{proof}
	We estimate the $\mathrm{L}_\infty$-norm by the sum of the absolute values of the Fourier coefficients and then use \cref{lemma:anova:trunc_coeff}
	\begin{align}
	\norm{f-\mathrm{T}_U f}{\mathrm{L}_\infty(\T^d)} &\leq \sum_{\k\in\Z^d} \abs{\fc{\k}{f} - \fc{k}{\mathrm{T}_U f} } = \sum_{\k\in\bigcup_{\substack{\u \subseteq \D \\ \u \notin U}}\Fud} \abs{\fc{\k}{f} } \nonumber \\
	&=  \sum_{\k\in\bigcup_{\substack{\u \subseteq \D \\ \u \notin U}}\Fud} \frac{w(\k)}{w(\k)}\abs{\fc{\k}{f} } \label{eq:star}\\&\leq \frac{1}{\min_{\k\in\bigcup_{\substack{\u \subseteq \D \\ \u \notin U}}\Fud} w(\k)} \norm{f}{\aw(\T^d)}. \nonumber
	\end{align}
	Employing the Cauchy-Schwarz inequality in \eqref{eq:star} instead of extracting the minimum yields 
	\begin{equation*}
		\norm{f-\mathrm{T}_U f}{\mathrm{L}_\infty(\T^d)} \leq \sqrt{\sum_{\k\in\bigcup_{\substack{\u \subseteq \D \\ \u \notin U}}\Fud} \frac{1}{w^2(\k)}} \,\,\norm{f}{\st(\T^d)}.
	\end{equation*}
	The condition $\{1/w(\k)\}_{\k\in\Z^d} \in \ell_2$ assures that the sum which appears in the bound is finite.
\end{proof}

In the following, we relate the truncation of $f$ by the operator $\mathrm{T}_{d_s}$ with the smoothness of $f$. To this end, we introduce the weights
\begin{equation}\label{eq:weights}
	w^{\alpha,\beta}(\k) \coloneqq \gamma_{\supp \k}^{-1} \, (1+\norm{\k}{1})^\alpha \prod_{s \in \supp \k} (1+\abs{k_s})^{\beta}
\end{equation}
with $\supp \k = \{ i \in \D \colon k_i \neq 0 \}$ and parameters $\beta \geq 0$, and $\alpha > -\beta$. The parameters $\alpha, \beta$, and the weight $\gamma_{\u}$, $\u\subset\D$, regulate the decay of the Fourier coefficients. Specifically, the parameter $\alpha$ is regulating the isotropic smoothness and $\beta$ the dominating mixed smoothness, cf. \cite{DuTeUl16}. Moreover, $\b\gamma$ controls the influence of the different dimensions. We choose a POD (\textit{product and order-dependent}) structure for $\gamma_{\u}$ such that \begin{equation}\label{eq:pod}
    \gamma_{\u} = \Gamma_{\au} \prod_{s \in \u} \gamma_s,
\end{equation}where $\b\Gamma \in (0,1]^{d}$ is nonincreasing and $\b\gamma = (\gamma_i)_{i=1}^{d} \in (0,1]^d$. The POD structure is motivated by the application of quasi-Monte Carlo methods for PDEs with random coefficients, cf.\ \cite{KuSchwSl12, GrKuNi14, KuNu16, GrKuNu18}. Similar weights for isotropic and dominating mixed smoothness have been considered in \cite{GriHa13, KaPoVo13, ByKaUlVo16}. Moreover, the Sobolev type spaces may also be referred to as weighted Korobov spaces, cf. \cite{SlWo01} for product weights and \cite{DiSlWaWo06} for general weights.

We now use the previously obtained bounds for general weight functions $w$ and derive results for the weights $w^{\alpha,\beta}$ from \eqref{eq:weights}. We focus on the subsets of ANOVA terms $U_{d_s}$ defined by a superposition threshold $d_s \in \D$.
\begin{corollary}\label{cor:bound:1}
	Let $f \in \wiener{w^{\alpha,\beta}}(\T^d)$ with weight function from \eqref{eq:weights} with POD structure \eqref{eq:pod}, $\beta \geq 0$, $\alpha > -\beta$, $\b\Gamma \in (0,1]^{d}$, and $\b\gamma \in (0,1]^d$. Then
	\begin{equation}\label{eq:trunc:bound_2}
	\frac{\norm{f-\mathrm{T}_{d_s} f}{\mathrm{L}_\infty(\T^d)}}{\norm{f}{\wiener{w^{\alpha,\beta}}(\T^d)}} \leq \Gamma_{d_s+1} \, (2+d_s)^{-\alpha}\,2^{-\beta (d_s+1)} \prod_{s=1}^{d_s+1} \gamma_s^{\ast} 
	\end{equation}
	where $\b{\gamma}^{\ast}$ is the non-increasing rearrangement of $\b\gamma$.
\end{corollary}
\begin{proof}
	We use \cref{theorem:anova:inf_error_1} and calculate the bound for the weight function $w^{\alpha,\beta,\b\gamma}$ by computing the minimum
	\begin{equation*}
		M \coloneqq \min_{\substack{\k\in\Z^d \\ \norm{\k}{0} > d_s}} \Gamma_{\left\| \k \right\|_0}^{-1}  (1+\left\| \k \right\|_1)^\alpha \prod_{s=1}^d (1+\abs{k_s})^\beta \prod_{s\in\supp \k}\gamma_s^{-1}.
	\end{equation*}
	Since $\b\Gamma$ is non-increasing by definition, $\Gamma_{d_s+1}^{-1}$ has to be equal to the smallest value. The frequencies in $\Fud$ have exactly $\au$ nonzero entries, therefore we get
	\begin{equation*}
		M = \Gamma_{d_s+1}^{-1} (1+d_s+1)^\alpha (1+1)^{\beta (d_s+1)} \min_{\substack{\k\in\Z^d \\ \norm{\k}{0} > d_s}} \prod_{s\in \u} \gamma_{s}^{-1}. 
	\end{equation*}
	The remaining product becomes minimal for the product of the $d_s+1$ smallest entries in $\b\gamma$ which yields the statement.
\end{proof}

\begin{lemma}\label{lemma:bounds:dim_estimate}
	Let $n \in \D$ and $\b\gamma \in (0,1]^d$. Then
	\begin{equation*}
	\sum_{\substack{\u \subseteq \D \\ \au = n}} \prod_{s\in\u} \gamma_s^2 \leq \left\| \b\gamma \right\|_2^{2n}.
	\end{equation*}
\end{lemma}
\begin{proof}
	We rewrite the sum as follows
	\begin{equation*}
	\sum_{\substack{\u \subseteq \D \\ \au = n}} \prod_{s\in\u} \gamma_s^2 = \sum_{i_1=1}^{d} \gamma_{i_1}^2 \sum_{i_2=i_1+1}^{d} \gamma_{i_2}^2 \cdots \sum_{i_n=i_{n-1}+1}^{d} \gamma_{i_n}^2.
	\end{equation*}
	Then every single sum can be estimated by $\left\| \b\gamma \right\|_2^{2}$, i.e., $$\sum_{i_j = i_{j-1}+1}^d \gamma_{i_j}^2 \leq \sum_{i_j = 1}^d \gamma_{i_j}^2 = \left\| \b\gamma \right\|_2^{2}$$ for $j \in \{2,3,\dots,d\}$ with equality for $j = 1$.
\end{proof}

\begin{corollary}\label{cor:bound:2}
	Let $f \in \mathrm{H}^{w^{\alpha,\beta}}(\T^d)$ with weight function from \eqref{eq:weights} with POD structure \eqref{eq:pod}, $\beta \geq 0$, $\alpha > -\beta$, $\b\Gamma \in (0,1]^{d}$, and $\b\gamma \in (0,1]^d$. Then
	\begin{equation}\label{eq:trunc:bound_1}
		\frac{\norm{f-\mathrm{T}_{d_s} f}{\Lt(\T^d)}}{\norm{f}{\mathrm{H}^{w^{\alpha,\beta}}(\T^d)}}\leq \Gamma_{d_s+1} \,(2+d_s)^{-\alpha}\,2^{-\beta (d_s+1)} \prod_{s=1}^{d_s+1} \gamma_s^{\ast}  
	\end{equation}
	where $\b{\gamma}^{\ast} = (\gamma_s^\ast)_{s=1}^d$ is the non-increasing rearrangement of $\b\gamma$. For functions with isotropic smoothness $\alpha = 0$ and dominating mixed smoothness $\beta > 1/2$ we have
	\begin{equation*}
		\frac{\norm{f-\mathrm{T}_{d_s} f}{\mathrm{L}_\infty(\T^d)}}{\norm{f}{\mathrm{H}^{w^{\alpha,\beta}}(\T^d)}} \leq \sqrt{\sum_{n=d_s+1}^{d} 2^{n} \Gamma_n^2 \left(\zeta(2\beta)-1\right)^{n} \left\| \b\gamma \right\|_2^{2n}}
	\end{equation*}	 
	where $\zeta$ is the Riemann zeta function. Exponential decay for $\Gamma_s$, i.e., $\Gamma_s = c^s$, $0 < c \leq 1$, such that the condition 
	\begin{equation}\label{eq:trunc:condition}
		\left\| \b\gamma \right\|_2 < \frac{1}{c\sqrt{2\zeta(2\beta)-2}}
	\end{equation}
	holds, yields the bound
	\begin{equation}\label{eq:trunc:inf_bound}
	\frac{\norm{f-\mathrm{T}_{d_s} f}{\mathrm{L}_\infty(\T^d)}}{\norm{f}{\mathrm{H}^{w^{\alpha,\beta}}(\T^d)}} \leq \frac{\left( c\left\| \b\gamma \right\|_2\sqrt{2\zeta(2\beta)-2} \right)^{d_s+1}}{\sqrt{1-2 c^{2} \left\| \b\gamma \right\|_2^{2}\left(\zeta(2\beta)-1\right) }}.
	\end{equation}	 
\end{corollary}
\begin{proof}
	The bound from statement \eqref{eq:trunc:bound_1} is a consequence of \cref{theorem:anova:l2_error} and can be calculated analogously to the proof of \cref{cor:bound:1}. For the second statement, we calculate the constant in the bound from \Cref{theorem:anova:inf_error_1}. We use \Cref{lemma:anova:sets} and the product structure of the weights $w^{\alpha,\beta}(\k)$ to obtain
	\begin{align*}
	\sum_{\k\in\bigcup_{\substack{\u \subseteq \D \\ \au > d_s}}\Fud} \frac{1}{w^2(\k)} &= \sum_{\substack{\u \subseteq \D \\ \au > d_s}} \sum_{\k\in\Fud} \frac{1}{\Gamma_{\au}^{-2}  \left(1+ \abs{k_s} \right)^{2\beta} \prod_{s\in\u} \gamma_s^{-2}} \\
	&= \sum_{\substack{\u \subseteq \D \\ \au > d_s}} \Gamma_{\u}^{2} \sum_{\k \in (\Z\setminus\{0\})^{\au}} \frac{1}{\left(\prod_{s\in\u} \gamma_s^{-2}\right)\left(\prod_{s=1}^{\au} (1+\abs{k_s})^{2\beta}\right)} \\
	&= \sum_{\substack{\u \subseteq \D \\ \au > d_s}} \Gamma_{\abs{\u}}^{2} \prod_{s\in\u} \gamma_s^2 \sum_{k\in\Z\setminus\{0\}} \frac{1}{ (1+\abs{k})^{2\beta}}.
	\end{align*}
	We find an explicit form by replacing the sums with the Riemann zeta function
	\begin{equation*}
	\prod_{s\in\u}   \gamma_s^2  \sum_{k\in\Z\setminus\{0\}} \frac{1}{ (1+k)^{2\beta}} = \prod_{s\in\u} 2  \gamma_s^2  \sum_{k\in\N} \frac{1}{ (1+k)^{2\beta}} = 2^{\au}  \left(\zeta(2\beta)-1\right)^{\au} \prod_{s\in\u} \gamma_s^2.
	\end{equation*}
	Applying \cref{lemma:bounds:dim_estimate} then gives us the upper bound
	\begin{equation*}
	\sum_{n=d_s+1}^{d} 2^{n} \Gamma_{n}^{2} \left(\zeta(2\beta)-1\right)^{n} \sum_{\substack{\u \subseteq \D \\ \au = n}} \prod_{s\in\u} \gamma_s^2 \leq \sum_{n=d_s+1}^{d} 2^{n} \Gamma_{n}^{2} \left(\zeta(2\beta)-1\right)^{n} \left\|\b \gamma \right\|_2^{2n}.
	\end{equation*}
	If we choose an exponential decay for $\Gamma_{n}$, i.e., $\Gamma_{n} \coloneqq c^{n}, 0 < c \leq 1$, the explicit upper bound becomes
	\begin{equation*}
	\sum_{n=d_s+1}^{d} 2^{n} c^{2n} \left(\zeta(2\beta)-1\right)^{n} \left\| \b\gamma \right\|_2^{2n} = \frac{q^{d_s+1}}{1-q} (1-q^{d-d_s})
	\end{equation*}
	where $q \coloneqq 2 c^{2} \left(\zeta(2\beta)-1\right) \left\| \b\gamma \right\|_2^{2}$ with $0 < q < 1$ because of the condition \eqref{eq:trunc:condition}.
\end{proof}

The bound in \Cref{cor:bound:1} and \eqref{eq:trunc:bound_1} in \Cref{cor:bound:2} are independent of the spatial dimensions $d$ of the functions $f$ as long as they have the same superposition threshold and the norm stays the same. This allows us to circumvent the curse of dimensionality here and use the ANOVA terms in $U_{d_s}$ for a superposition threshold $d_s \in \D$. The bound \eqref{eq:trunc:inf_bound} can also be considered for $d \rightarrow \infty$. The dependence on the dimension $d$ is contained within the norm $\left\| \b\gamma \right\|_2^{2}$. Choosing a square-summable sequence $\{\gamma_\ell\}_{\ell \in \N}$ results in an upper bound for $\left\| \b\gamma \right\|_2$ for any $d \rightarrow \infty$. In this case the bound can be made independent of $d$ by the condition \eqref{eq:trunc:condition}. 

\Cref{fig:trunc:bounds} shows the different bounds for weights $w^{\alpha,\beta}$ with $\b\gamma = (1/s)_{s=1}^9$ and $\b\Gamma = (\pi^{-s}\sqrt{3}^s )_{s=1}^9$, see \eqref{eq:weights}. With regard to the superposition dimension $\dsuper$ for $\sobolev{w^{\alpha,\beta}}(\T^d)$, cf.~\eqref{eq:anova:superposition_specific}, one may interpret this as follows: Given $f \in \sobolev{w^{\alpha,\beta}}(\T^d)$, the value $\varepsilon(\alpha,\beta) \in (0,1)$ of the bound in part (a) of \Cref{fig:trunc:bounds} tells us that for $\delta = 1 - \varepsilon(\alpha,\beta)^2$ the superposition dimension $\dsuper$ is smaller or equal to the superposition threshold $d_s = 3$, e.g., $\varepsilon(0,1) \approx 0.0008$ and therefore $\delta = 0.99999936$.

\begin{figure}
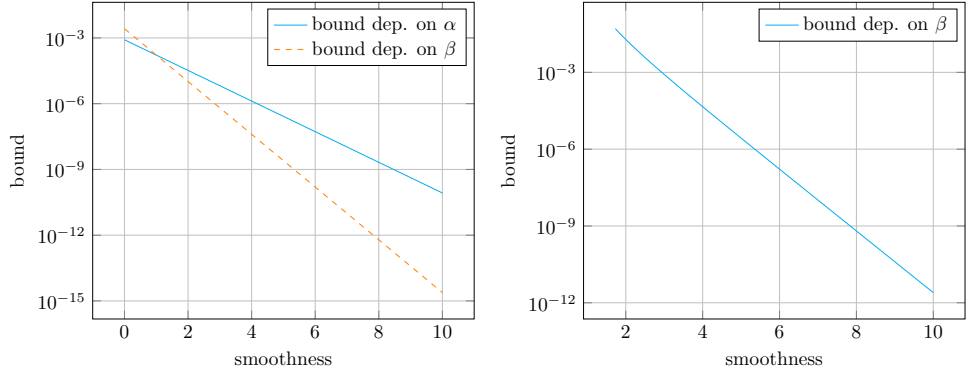
	
	\subfloat[bounds \eqref{eq:trunc:bound_1}, \eqref{eq:trunc:bound_2} for $0 \leq \alpha \leq 10$, $\beta = 1$ (solid), and for $\alpha = 1$, $0 \leq \beta \leq 10$ (dashed)]{

	}
	\caption{Decay of errors from \eqref{eq:trunc:bound_1}, \eqref{eq:trunc:inf_bound}, and \eqref{eq:trunc:bound_2} in relation to their isotropic smoothness $\alpha$ and dominating-mixed smoothness $\beta$ with $d = 9$, $d_s = 3$, dimension dependent coefficients $\b\gamma = (1/s)_{s=1}^9$ and order dependent coefficients $\b\Gamma = (\pi^{-s}\sqrt{3}^s )_{s=1}^9$.}
	\label{fig:trunc:bounds}
\end{figure}

\section{ANOVA approximation method}\label{sec:anova:approxmethod}

We consider the general problem of approximating a periodic function $\fun{f}{\T^d}{\C}$ given certain function evaluations of $f$. Specifically, we distinguish two approximation scenarios -- black-box approximation and scattered data approximation. In the case of black-box approximation, we are able to evaluate $f$ at any given point $\x \in \T^d$. Since the evaluations come at a certain cost, we aim to keep them minimal or require a certain trade-off. For scattered data approximation we have a finite set of nodes $X \subset \T^d$ and know the function values $\b y = (f(\x))_{\x\in X}$. Here, one cannot add more nodes to $X$ or choose the locations of the nodes. Both scenarios have a high relevance for problems in various applications. 

In this section, we consider an approximation scheme for high-dimensional, periodic functions of a low-dimensional structure, i.e., functions with a small superposition dimension $\dsuper \in \D$ for a $\delta \in (0,1]$ that is close to one, cf.~\eqref{eq:anova:superposition_specific}. In this case the truncation by $\mathrm{T}_{d_s}$ with a small superposition threshold $d_s \in \D$ will be effective. It has been observed that functions in many practical applications belong to such a class, see e.g.~\cite{CaMoOw97}. In \Cref{sec:trunc} we have considered errors for functions of dominating-mixed and isotropic smoothness defined trough the decay of the Fourier coefficients and therefore obtained an upper bound for the modified superposition dimension $\dsuper$ from \eqref{eq:anova:superposition_specific}. Considering \cref{fig:trunc:bounds}, we know that e.g.~POD weights lead to a decay such that the functions are of a low-dimensional structure.

The approximation scheme can be viewed in both approximation scenarios although the details are different. We work for now with the node set $X$ as well as function evaluations $\b y$ and keep in mind that $X$ may also be chosen if we are in the black-box case. The first step is to reduce the ANOVA decomposition to the terms in $U_{d_s}$, i.e., we approximate \begin{equation*}
    f \approx \Tds f = \sum_{\u \in U_{d_s}} f_{\u}.
\end{equation*} The Fourier coefficients $\fc{\k}{\Tds f}$ can only be nonzero if the frequency $\k$ is at most $d_s$-sparse, i.e., $\norm{\k}{0} \leq d_s$, see \cref{cor:trunc_coeff:ds}. Based on this, we aim to approximate $f$ by a Fourier partial sum $S_I f$ with a finite index set \begin{equation}\label{eq:index_set}
    I \subset \left\{ \k \in \Z^d \colon \supp \k \in U_{d_s} \right\}.
\end{equation} The challenge is to determine an appropriate index set $I$. To this end, we employ a special scheme to determine frequency locations based on the ANOVA terms and an importance ranking on them.

We call the first step \textit{active set detection} and its aim is to determine an importance ranking on the terms $f_{\u}$ with $\u \in U_{d_s}$ based on the global sensitivity indices $\gsi{\u}{f}$, cf.~\eqref{eq:anova:gsi}. This information is also highly relevant to interpret relations in our data $X$ and $\y$. 

Based on the sensitivity indices we build an active set of ANOVA terms $U \subset U_{d_s}$. This relates to the importance of frequencies and therefore information on how to choose the index set $I$ from \eqref{eq:index_set}. Reducing the number of ANOVA terms and in turn the number of frequencies leads to a reduction of the model complexity. The effects of overfitting are therefore lessened. In \Cref{sec:approx:active} we consider the details of the active set detection and in \Cref{sec:approx:final} the approximation with an active set as well as approximation errors.

\subsection{Active set detection}\label{sec:approx:active}

The method assumes that the underlying function $f$ is of a low-dimensional structure, i.e., $f \approx \Tds f$ for some superposition threshold $d_s \in \D$. The goal in the active set detection step is to determine an importance ranking for the ANOVA terms. In order to do this, we choose an appropriate search index set. Since we have no a-priori knowledge about the importance of the ANOVA terms or the smoothness of the function $f$, we work with order-dependent finite index sets $I_0 = \{ 0 \}, I_1 \subset (\Z\setminus\{0\}), \dots, I_{d_s} \subset (\Z\setminus\{0\})^{d_s}$. This achieves that two ANOVA terms $f_{\u}$ and $f_{\v}$ with $\au = \av$ are supported on equivalent index sets. We then use the projection operator \begin{equation}\label{eq:pre:proj}
	\mathrm{P}_{\u} I \coloneqq \{ \k\in\Z^d \colon \k_{\u} \in I, \k_{\uc} = \b 0 \}
\end{equation} to project the index sets and obtain \begin{equation}\label{eq:IU}
I(U_{d_s}) = \bigcup_{\u\in U_{d_s}}\mathrm{P}_{\u}I_{\au}. 
\end{equation} This leads to the approximation by a Fourier partial sum \begin{equation}\label{eq:approx:partial}
    f(\x) \approx \Tds f (\x) \approx S_{I(U_{d_s})} f (\x) = \Fseries{\k}{I(U_{d_s})}{\fc{\k}{f}}{\x}.
\end{equation}

The Fourier coefficients $\fc{\k}{f}$ in \eqref{eq:approx:partial} are unknown and we aim to determine approximations for them from the data $X$ and $\b y$. To this end, we consider the least-squares problem \begin{equation}\label{eq:approx:ls_ds}
    \bfh_{\text{sol}} = \argmin_{\bfh \in \C^{\abs{I(U_{d_s})}}} \norm{\y - \F_{I(U_{d_s})} \bfh}{2}^2
\end{equation}
with Fourier matrix $\F_{I(U_{d_s})} = \left( \e^{2\pi\i\k\cdot\x} \right)_{\x\in X,\k\in I(U_{d_s})}$. If the Fourier matrix has full rank, the elements of the solution vector $\bfh_{\text{sol}} = (\hat{f}_{\k})_{\k \in I(U_{d_s})}$ are the unique least-squares approximation to the Fourier coefficients, i.e., $\hat{f}_{\k} \approx \fc{\k}{f}$, with respect to $X$ and $\y$. Depending on the approximation scenario, there are different methods of solving least-squares problems of the type \eqref{eq:approx:ls_ds}. We refer to \Cref{sec:approx:ls} for details.

We use the approximate Fourier coefficients $\hat{f}_{\k}$ to build the approximate Fourier partial sum \begin{equation}\label{eq:approx:initial}
    S_{I(U_{d_s})} f (\x) \approx S_{I(U_{d_s})}^X f (\x) = \Fser{\k}{I(U_{d_s})}{\hat{f}_{\k}}{\x}
\end{equation}
which provides an initial approximation to the function $f$. In order to achieve a Fourier matrix $\F_{I(U_{d_s})}$ with full rank and combat the effects of overfitting, we may need to severely limit the number of frequencies in the order-dependent sets $I_1$, $I_2$, \dots, $I_{d_s}$. Details on this will be considered in the following subsections for the specific approximation scenarios.

In order to determine an importance ranking on the ANOVA terms, we assume that the global sensitivity indices of $S_{I(U_{d_s})}^X f$ and $f$ behave similarly, i.e., it holds that \begin{equation}\label{assumption}
    \gsi{\u_1}{S_{I(U_{d_s})}^X f} \leq \gsi{\u_2}{S_{I(U_{d_s})}^X f} \Longrightarrow \gsi{\u_1}{f} \leq \gsi{\u_2}{f}
\end{equation} for $\u_1, \u_2 \in U_{d_s}$. This allows us to use a threshold vector $\b\varepsilon \in [0,1]^{d_s}$ to define an active set of ANOVA terms that only contains the \textit{important} terms with respect to $\b\varepsilon$
\begin{equation}\label{eq:approx:active_set}
	U_{X,\y}^{(\b\eps)} \coloneqq \{ \v \subset \D \colon \exists \u \in U_{d_s} \colon \v \subset \u ~\text{and}~ \gsi{\u}{S_{I(U_{d_s})}^X f} > \varepsilon_{\au} \}.
\end{equation}
The inclusion condition \eqref{eq:trunc_inclusion} is fulfilled by definition. We reduce the ANOVA decomposition to this set of terms to determine an approximation for $f$ in \Cref{sec:approx:final}.

\subsection{Least-squares approximation}\label{sec:approx:ls}

In this section, we discuss the solution of least-squares problems of the form \begin{equation}\label{eq:approx:gen_ls}
    \min_{\bfh \in \C^{\abs{I(U)}}} \norm{\y - \F_{I(U)} \bfh}{2}^2
\end{equation} 
with a Fourier matrix $\F_{I(U)} = \left( \e^{2\pi\i\k\cdot\x} \right)_{\x\in X,\k\in I(U)}$. Here, $U$ is an arbitrary subset of ANOVA terms and for each term we have a given finite frequency index set $I_{\u} \subset (\Z\setminus\{0\})^{\au}$. The set \begin{equation}\label{set:IUfull}
	I(U) = \bigcup_{\u\in U}\mathrm{P}_{\u}I_{\u}
\end{equation} is obtained through the projections \eqref{eq:pre:proj}. 

The following remark shows that the Fourier matrix can be structured with respect to the ANOVA terms. Moreover, we can decompose the matrix-vector multiplications with both, $\F_{I(U)}$ and its adjoint $\F_{I(U)}^\ast$. \begin{remark}\label{rem:approx:struct}
	Let $\F_{I(U)}$ be a Fourier matrix with respect to a node set $X$ and an index set $I(U)$ with a subset of ANOVA terms $U \subset \mathcal{P}(D)$ and index sets $I_{\u} \subset (\Z\setminus\{0\})^{\au}$, $\u \in U$. Then \begin{equation*}
	\b F \bfh = \left( \b{F}_{\u_1} \,\,\, \b{F}_{\u_2} \,\,\, \cdots \,\,\, \b{F}_{\u_n} \right) \bfh
	\end{equation*} 
	where $\u_1, \u_2, \dots, \u_n$ with $n = {\abs{U}}$ is a numbering of the subsets of coordinate indices in $U$ such that $\hat{\b f} = \left( \hat{\b f}_{\u_1} \,\,\, \hat{\b f}_{\u_2} \,\,\, \cdots \,\,\, \hat{\b f}_{\u_n} \right)^\top$. The Fourier matrices are $\b{F}_{\u} = \left( \e^{2\pi\i\b\ell\cdot\x_{\u}} \right)_{\x\in X,\b\ell\in I_{\u}}$. The matrix-vector product with $\b F$ can therefore be decomposed as
	\begin{equation*}
	\b F \bfh = \sum_{\u \in U} \b{F}_{\u} \bfh_{\u}
	\end{equation*} with vector components $\bfh_{\u}$. For the adjoint product $\b{F}^\ast \b f$ with a vector $\b f \in \C^{\abs{X}}$ we obtain the result $\hat{\b a} \in \C^{\abs{I(U)}}$ by computing the products \begin{equation*}
	\hat{\b a}_{\u} = \b{F}_{\u}^\ast \b f,\quad \forall\u \in U.
	\end{equation*} 
	Then we have the result vector $\hat{\b a} = \left( \hat{\b a}_{\u_1} \,\,\, \hat{\b a}_{\u_2} \,\,\, \cdots \,\,\, \hat{\b a}_{\u_n} \right)^\top$.
\end{remark}

\subsubsection{Black-box scenario}\label{sec:approx:blackbox} 

In the case of black-box approximation, i.e., the set $X$ can be chosen, we have to determine an appropriate special discretization for index sets of the type $I(U)$. Here, we have different possibilities. One might think of rank-1 lattices that have been used for integration before, see e.g.~\cite{JoKuSl13}, and approximation, see e.g.~\cite{KaPoVo13, KaVo19}. For a general introduction to lattice rules, we refer to \Cref{sec:pre:lattice}. Sparse grid sampling related to the Smolyak algorithm is a further possibility, cf.~\cite{Gr05, He03, HeLe11, Holtz11}.

In the following, we focus on using reconstructing single rank-1 lattice for function approximation. If we have a reconstructing single rank-1 lattice $\Lambda(\b z, M, I(U)) \subset \Z^d$ for a generating vector $\b z \in \Z^d$ and size $M \in \N$ with respect to an index set $I(U)$, then \begin{equation}\label{eq:approx:lat_prop}
    \F_{I(U)}^\ast \F_{I(U)} = M \cdot \mathrm{I}
\end{equation} with $\mathrm I$ the identity matrix, see \cite[Chapter 8.2]{PlPoStTa18}. Then the solution to problem \eqref{eq:approx:gen_ls} is unique and given by the multiplication of the Moore-Penrose inverse $\F_{I(U)}^\dagger$ with $\b y$, see e.g.~\cite{Bj96}. Through the property \eqref{eq:approx:lat_prop} the Moore-Penrose inverse is simplified to \begin{equation}\label{eq:approx:mpinv}
    \F_{I(U)}^\dagger = \frac{1}{M} \F_{I(U)}^\ast,
\end{equation} i.e., a multiplication with the adjoint matrix. This allows us to efficiently compute approximations for the Fourier coefficients of $f$ if the nodes form a reconstructing rank-1 lattice.

It remains the issue of determining such a reconstructing rank-1 lattice given an index set of type $I(U)$. In \cite[Theorem 8.16]{PlPoStTa18} it was shown that reconstructing lattices exist if the lattice size $M$ is sufficiently large. Since the evaluations of $f$ come at a certain cost, it is necessary to consider the lattice size for our special types of index sets which we do in the following.

An important quantity to get estimations on the lattice size is the difference set $\D(I(U))$ from \eqref{eq:pre:di} since the result \cite[Theorem 8.16]{PlPoStTa18} tells us that there exists a reconstructing rank-1 lattice with prime cardinality \begin{equation*}
\abs{I(U)} \leq M \leq \abs{\D(I(U))}.
\end{equation*} In the following, we proof properties and show estimates on the cardinality of both $I(U)$ and $\D(I(U))$.

\begin{lemma}\label{lemma:est_di}
	Let $U \subset \mathcal{P}(\D)$ be a subset of ANOVA terms and $I_{\u} \subset \Z^{\au}$, $\u \in U$, finite symmetric frequency sets. Then we have \begin{equation*}
		\D(I(U)) = \bigcup_{\substack{\u \in U \\ \v \subset \u}} \{ \k-\b h \colon \k \in \mathrm{P}_{\u}I_{\u}, \b h \in \mathrm{P}_{\v}I_{\v} \}.
	\end{equation*}
\end{lemma}
\begin{proof}
	It is easy to see that $\bigcup_{\substack{\u \in U \\ \v \subset \u}} \{ \k-\b h \colon \k \in \mathrm{P}_{\u}I_{\u}, \b h \in \mathrm{P}_{\v}I_{\v} \} \subset \D(I(U))$ since $\mathrm{P}_{\u}I_{\u} \subset I(U)$ for every $\u\in U$ and $\v\in U$ for all $\v\subset\u\in U$ due to \eqref{eq:trunc_inclusion}. In order to show the other inclusion we take an element $\b\ell \in \D(I(U))$. By the uniqueness property of the ANOVA decomposition we know that there exists $\u,\v \in U$ such that $\b\ell = \k - \b h$ with $\b k \in \mathrm{P}_{\u}I_{\u}$ and $\b h \in \mathrm{P}_{\v}I_{\v}$. Taking the symmetry of the index sets $I_{\u}$ into account, we have proven the statement.
\end{proof}
The following lemma gives an estimate for the size of the difference set of index sets of type $I(U)$ if there exists an upper bound on the cardinality of the term dependent sets $I_{\u}$.
\begin{lemma}\label{lem:approx:diff}
	Let $U$ be a subset of ANOVA terms and $I_{\u} \subset (\Z\setminus\{0\})^{\au}, \u \in U,$ symmetric frequency sets. Then the cardinality of the difference set of $I(U)$ is bounded by \begin{equation}\label{eq:approx:diff_est}
		\abs{\D(I(U))} \leq \sum_{\u \in U} \sum_{\v\subset\u} \abs{I_{\u}} \abs{I_{\v}} \leq 2^{\max_{\u\in U} \au} \abs{U} \max_{\u \in U} \abs{I_{\u}}^2.
	\end{equation} 
\end{lemma}
\begin{proof}
	We estimate the cardinality of the difference set by applying \Cref{lemma:est_di} $$ \abs{\D(I(U))} \leq \sum_{\u \in U} \sum_{\v\subset\u} \abs{I_{\u}} \abs{I_{\v}}.$$ Here, we do not have equality since the union in \Cref{lemma:est_di} is not necessarily disjoint. Applying the upper bound on the cardinality of the sets $I_{\u}$, we arrive at \begin{align*} \sum_{\u \in U} \sum_{\v\subset\u} \abs{I_{\u}} \abs{I_{\v}} &\leq \sum_{\u \in U} \sum_{\v\subset\u} \max_{\u \in U} \abs{I_{\u}}^2 \leq \max_{\u \in U} \abs{I_{\u}}^2 \, 2^{\max_{\u\in U} \au} \sum_{\u \in U} 1 \\ &\leq 2^{\max_{\u\in U} \au} \abs{U} \max_{\u \in U} \abs{I_{\u}}^2.\end{align*}
\end{proof}
\begin{remark}
    The cardinality of $U_{d_s}$ is bounded by $(\e\cdot d/d_s)^{d_s}$, see \cref{lem:polycard}. Therefore the estimate in \eqref{eq:approx:diff_est} becomes \begin{equation*}
        \abs{\D(I(U_{d_s}))} \leq \left(\frac{2\e\cdot d}{d_s}\right)^{d_s} \max_{\u \in U} \abs{I_{\u}}^2.
    \end{equation*}
\end{remark}

In the following, we consider special term-dependent frequency index sets of the structure \begin{equation}\label{eq:approx:freq_weight}
    I_{\u} \coloneqq \left\{ \b\ell \in (\Z\setminus\{0\})^{\au} \colon w(\k) \leq N_{\u} \text{ for } \k \in \Z^d \text{ with } \k_{\u} = \b\ell,\,\k_{\uc} = \b 0 \right\}
\end{equation} with a subset of coordinate indices $\emptyset \neq \u \subset \D$, a weight function $\fun{w}{\Z^d}{[1,\infty)}$ and cut-off $N_{\u} \in \N$. For a given subset of ANOVA terms $U \subset \mathcal{P}(\D)$ we estimate the cardinalities of both, $I(U)$ and the difference set $\D(I(U))$.

\begin{lemma}\label{lem:approx:IUcard}
	Let $U \subset \mathcal{P}(\D)$ be a subset of ANOVA terms, $I_\emptyset = \{\b 0\}$, and $I_{\u}$, $\emptyset \neq \u \in U$, finite frequency sets as in \eqref{eq:approx:freq_weight} for a weight function $\fun{w}{\Z^d}{[1,\infty)}$ and $N_{\u} \in \N$. Moreover, let $\fun{h_{\min}}{\N}{[1,\infty)}$ and $\fun{h_{\max}}{\N}{[1,\infty)}$ be functions such that \begin{equation*}
		c\,h_{\min}(N_{\u_{\min}}) \leq \min_{\u \in U\setminus\{\emptyset\}} \abs{I_{\u}} \text{ and } \max_{\u \in U} \abs{I_{\u}} \leq C\, h_{\max}(N_{\u_{\max}})
	\end{equation*} with $\u_{\min} = \argmin_{\u\in U\setminus\{\emptyset\}} \abs{I_{\u}}$, $\u_{\max} = \argmax_{\u\in U} \abs{I_{\u}}$, and $0 < c \leq C$. Then we have for the asymptotic behavior of the cardinality of $I(U)$ \begin{equation*}
		c\,h_{\min}(N_{\u_{\min}}) \leq \frac{\abs{I(U)}}{\abs{U}} \leq C\, h_{\max}(N_{\u_{\max}}).
	\end{equation*} The constants do not depend on the spatial dimension $d$.
\end{lemma}
\begin{proof}
	Since the projected sets $\mathrm{P}_{\u} I_{\u}$, $\u\in U$, are disjoint, we have \begin{equation*}
		\abs{I(U)} = \sum_{\u\in U} \abs{I_{\u}}.
	\end{equation*} In order to show the upper bound, we estimate the cardinality of each index set by $h_{\max}$
	\begin{equation*}
		\sum_{\u\in U} \abs{I_{\u}} \leq \sum_{\u\in U} C\, h_{\max}(N_{\u_{\max}}) \leq C\, h_{\max}(N_{\u_{\max}}) \sum_{\u\in U} 1 = \abs{U}\, C\, h_{\max}(N_{\u_{\max}}).
	\end{equation*} The lower bound follows with similar arguments.
\end{proof}

\begin{corollary}\label{cor:approx:diff_card}
    Let $U \subset \mathcal{P}(\D)$ be a subset of ANOVA terms, $I_\emptyset = \{\b 0\}$, and $I_{\u}$, $\emptyset \neq \u \in U$, finite symmetric frequency sets as in \eqref{eq:approx:freq_weight} for a weight function $\fun{w}{\Z^d}{[1,\infty)}$ and $N_{\u} \in \N$. Moreover, let $h_{\max}$ be a function as in \Cref{lem:approx:IUcard}. Then
    \begin{equation*}
        \abs{\D(I(U))} \leq C^2 \, 2^{\max_{\u\in U} \au} \abs{U} h_{\max}^2(N_{\u_{\max}}).
    \end{equation*}
\end{corollary}
\begin{proof}
    The corollary is a direct consequence of \cref{lem:approx:diff} and \cref{lem:approx:IUcard}.
\end{proof}

We may apply \cite[Algorithm 8.17]{PlPoStTa18} to construct the reconstructing rank-1 lattice $\Lambda(\b z, M, I(U))$ via a component-by-component approach. 
Choosing the set $X = \Lambda(\b z, M, I(U))$ as sampling nodes yields a Moore-Penrose inverse of type \eqref{eq:approx:mpinv} and we are able to compute the solution to \eqref{eq:approx:gen_ls} by multiplying with the adjoint Fourier matrix. This computation can be done efficiently using a lattice fast Fourier transform or LFFT, see \cite[Section 8.2.2]{PlPoStTa18}.

\subsubsection{Scattered data scenario}\label{sec:approx:scattered}

In this section, we consider the scenario of scattered data approximation, i.e., we have a fixed set of nodes $X \subset \T^d$. Here, we aim to solve the least-squares problem \eqref{eq:approx:gen_ls} with the iterative LSQR method \cite{PaSa82}. Specifically, we are interested in the matrix-free variant, i.e., we do not have to construct the system matrix $\F_{I(U)} \in \C^{\abs{X},\abs{I(U)}}$ explicitly. The curse of dimensionality would quickly lead to the size of the matrix becoming intractable. The matrix-free variant requires two algorithms, one which takes a vector $\b a \in \C^{\abs{I(U)}}$ as an input and returns the result of the matrix-vector multiplication $\F_{I(U)} \b a$ and one that takes $\hat{\b a} \in \C^{\abs{X}}$ as an input and returns the result of $\F^{\ast}_{I(U)} \hat{\b a}$. If we take \cref{rem:approx:struct} into account, it is only necessary to provide algorithms for fast multiplication with Fourier matrices $\F_{I_{\u}} \in \C^{\abs{X},\abs{I_{\u}}}$, $\u \in U$.

The existence of such algorithms depends on the choice of the specific index sets $I_{\u}$. For full grids, i.e., frequency sets of the type $$ I_{\u} = G_{N}^{\u} = \left\{ \k\in\Z^{\au}: -\frac{N_{\u}}{2} \leq k_i \leq \frac{N_{\u}}{2} - 1, i = 1,2,\dots,\au \right\}, N_{\u} \in 2\N, $$ the non-equispaced fast Fourier transform (NFFT) was introduced in \cite{KeKuPo09}. Moreover, for hyperbolic cross index sets of the form \begin{equation*}
    I_{\u} = H_n^{\au} = \bigcup_{\substack{\b j \in \N_0^{\au} \\ \norm{\b j}{1} = n}} \hat{G}_{\b j}
\end{equation*} with $\hat{G}_{\b{n}} = \times_{s=1}^{\au} \hat{G}_{n_s}$ and $\hat{G}_{n_s} = (-2^{n_s-1}, 2^{n_s-1}]^{\au} \cap \Z$, we have the non-equispaced hyperbolic cross fast Fourier transform (NHCFFT), cf.~\cite{DoKuPo09}.

\subsection{Approximation with active set}\label{sec:approx:final}

Now that we have obtained the active set $U_{X,\y}^{(\b\eps)}$ from \eqref{eq:approx:active_set}, we aim to construct an approximation using only these ANOVA terms. The global sensitivity indices $\gsi{\u}{S_{I(U_{d_s})}^X f }$ calculated from the approximation $S_{I(U_{d_s})}^X f $ in \eqref{eq:approx:initial} provide us with a basis to choose term-dependent frequency index sets $I_{\u} \subset (\Z\setminus\{0\})^{\au}$, $\emptyset \neq \u \in U_{X,\y}^{(\b\eps)}$. A higher sensitivity index suggests that the term is more important to the function and therefore a larger corresponding index set could be advisable. 

We project the index sets as before to obtain $I(U_{X,\y}^{(\b\eps)})$, see \eqref{set:IUfull}. Note that in general and depending on the threshold $\b\epsilon$, we have reduced the number of frequencies significantly. This is a sensible measure to reduce the effects of overfitting. Now, we approximate $f$ by the Fourier partial sum \begin{equation*}
    f(\x) \approx \mathrm{T}_{U_{X,\y}^{(\b\eps)}} f (\x) \approx S_{I(U_{X,\y}^{(\b\eps)})} f (\x) = \Fser{\k}{I(U_{X,\y}^{(\b\eps)})}{\fc{\k}{f}}{\x}.
\end{equation*} The Fourier coefficients $\fc{\k}{f}$ are again unknown and we determine them by least-squares approximation from $X$ and $\y$. The unique solution is given by \begin{equation}\label{eq:approx:last}
    \bfh_{\text{sol}} = \argmin_{\bfh \in \C^{\abs{I(U_{X,\y}^{(\b\eps)})}}} \norm{\y - \F_{I(U_{X,\y}^{(\b\eps)})} \bfh}{2}^2
\end{equation} if the Fourier matrix $\F_{I(U_{X,\y}^{(\b\eps)})} = \left( \e^{2\pi\i\k\cdot\x} \right)_{\x\in X,\k\in I(U_{X,\y}^{(\b\eps)})}$ has full rank. Details on how to solve this system for scattered data and black-box approximation can be found in \cref{sec:approx:ls}. We use the elements of the solutions vector $\bfh_{\text{sol}} = (\hat{f}_{\k})_{\k \in I(U_{X,\y}^{(\b\eps)})}$ to form the approximate Fourier partial sum and our solution \begin{equation*}
    f(\x) \approx S_{I(U_{X,\y}^{(\b\eps)})} f (\x) \approx S_{I(U_{X,\y}^{(\b\eps)})}^X f (\x) = \Fser{\k}{I(U_{X,\y}^{(\b\eps)})}{\hat{f}_{\k}}{\x}.
\end{equation*} The following algorithm summarizes the proposed method.

\begin{algorithm}[ht]
	\vspace{2mm}
	\begin{tabular}{ l l l }
		\textbf{Input:} & $X \subset \T^d$ & finite node set \\
		& $\b y = (f(\x))_{\x\in X}$ & function values \\
		& $d_s \in \D$ & superposition threshold \\
	\end{tabular}
	\begin{algorithmic}[1]
	    \STATE{Choose finite order-dependent search sets $I_1 \subset \Z\setminus\{0\}, \dots, I_{d_s} \subset (\Z\setminus\{0\})^{d_s}$.}
	    \STATE{Compute solution of least-squares problem \eqref{eq:approx:ls_ds}.}
	    \STATE{$\bfh_{\text{sol}} = (\hat{f}_{\k})_{\k \in I(U_{d_s})} \leftarrow \argmin_{\bfh \in \C^{\abs{I(U_{d_s})}}} \norm{\y - \F_{I(U_{d_s})} \bfh}{2}^2$}
	    \STATE{Compute global sensitivity indices for approximation $S_{I(U_{d_s})}^{X} f$ using \eqref{eq:anova:gsi}.}
	    \STATE{$\varrho(\u,S_{I(U_{d_s})}^{X} f) \leftarrow \frac{\norm{(S_{I(U_{d_s})}^{X} f)_{\u}}{\Lt(\T^{d})}^2}{\norm{S_{I(U_{d_s})}^{X} f}{\Lt(\T^{d})}^2 - \abs{\fc{0}{S_{I(U_{d_s})}^{X} f}}^2}$, $\u \in U_{d_s}$}
        \STATE{Choose threshold vector $\b\eps \in [0,1]^{d_s}$ and build active set.}
        \STATE{$U_{X,\b y}^{(\b\varepsilon)} \leftarrow \left\{ \v\subset\D \colon \exists\u\in U_{d_s}\colon\v \subset \u \quad\text{and}\quad \gsi{\u}{S_{I(U_{d_s})}^X f} > \varepsilon_{\au} \right\}$}
        \STATE{Use information from global sensitivity indices to choose finite index sets $I_{\u} \subset (\Z\setminus\{0\})^{\au}$ per ANOVA term in $U_{X,\b y}^{(\varepsilon)}$.}
        \STATE{Compute solution of least-squares problem \eqref{eq:approx:last}.}
        \STATE{$\bfh_{\text{sol}} = (\hat{f}_{\k})_{\k \in I(U_{X,\b y}^{(\b\varepsilon)})} \leftarrow \argmin_{\bfh \in \C^{\abs{I(U_{X,\y}^{(\b\eps)})}}} \norm{\y - \F_{I(U_{X,\y}^{(\b\eps)})} \bfh}{2}^2$}
	\end{algorithmic}
	\begin{tabular}{ l l l }
		\textbf{Output:} & $\hat{f}_{\k} \in \C, \k \in I(U_{X,\b y}^{(\b\varepsilon)})$ & approximations to Fourier \\
		 & & coefficients $\fc{\k}{f}$ \\
		& $\varrho(\u,S_{I(U_{d_s})}^{X} f) \in [0,1], \u \in U_{d_s}$ & global sensitivity indices of $S_{I(U_{d_s})}^{X} f$ \\
		& & or \textit{importance ranking} on the terms
	\end{tabular}
	\caption{ANOVA Approximation Method}
	\label{alg}
\end{algorithm}

\section{Error analysis}\label{sec:anova:error}

The error of our approximation method measured in the norm of some space $H\subset\L_2(\T^d)$ can be decomposed into multiple components by the triangle inequality \begin{equation*}
    \norm{f - S_{I(U)}^X f}{H} \leq \underbrace{\norm{f - \mathrm{T}_{U} f}{H}}_{\text{ANOVA truncation error}}+ \underbrace{\norm{\mathrm{T}_{U} f - S_{I(U)}^X f}{H}}_{\text{approximation error}}
\end{equation*} for an active set of ANOVA terms $U \subset U_{d_s}$ with superposition threshold $d_s \in \D$. We distinguish between the ANOVA truncation error and the approximation error. Here, the analysis of the ANOVA truncation error is independent of the concrete approximation problem \eqref{eq:approx:last} and the scenario (scattered data or black-box).

\subsection{ANOVA truncation error}

The ANOVA truncation error is related to the truncation of the ANOVA decomposition to the set $U_{d_s}$ with superposition threshold $d_s \in \D$ and the active set $U \subset U_{d_s}$. We can separate the ANOVA truncation error as follows \begin{equation}\label{eq:anv:trunc:err}
    \norm{f - \mathrm{T}_{U} f}{H} \leq \underbrace{\norm{f - \mathrm{T}_{d_s} f}{H}}_{\text{truncation by }d_s} + \underbrace{\norm{\mathrm{T}_{d_s} f - \mathrm{T}_{U} f}{H}}_{\text{active set truncation}}.
\end{equation} Here, we bring $\mathrm{T}_{d_s}$ in with the aim to relate the error to our function class of low-order interactions, see \eqref{eq:trunc_err}. To control the second term, we require assumptions on the sensitivity indices of the ANOVA terms in $U_{d_s}\setminus U$. Since the error is only related to the structure of the function it can be considered independently of any specific approximation scenario like black-box or scattered data approximation. We show bounds for this error in the case that $f$ is an element of a Sobolev type space $\st(\T^d)$ or a Wiener algebra $\aw(\T^d)$ and $H$ is $\Lt(\T^d)$ or $\mathrm{L}_\infty(\T^d)$.

\begin{theorem}\label{thm:av_trunc_error:1}
    Let $f \in \st(\T^d)$ with a weight function $\fun{w}{\Z^d}{[1,\infty)}$ and superposition dimension $\dsuper$, see \eqref{eq:anova:superposition_specific}, for a $\delta \in (0,1)$. If there exists a subset of ANOVA terms $U \subset U_{\dsuper}$ such that $$\gsi{\u}{f} = \frac{\va{f_{\u}}}{\va{f}} < \varepsilon, \,\varepsilon > 0,$$ for every $\u \in U_{\dsuper} \setminus U$ then \begin{equation*}
        \frac{\norm{f - \mathrm{T}_{U} f}{\Lt(\T^d)}}{\norm{f}{\st(\T^d)}} \leq \sqrt{1-\delta} + \sqrt{\abs{U_{\dsuper} \setminus U} \varepsilon}.
    \end{equation*}
\end{theorem}
\begin{proof}
    The ANOVA truncation error can be separated as in \eqref{eq:anv:trunc:err}. We prove an upper bound for the active set truncation. With Parseval's equality and the assumption on the global sensitivity indices, we estimate \begin{equation}\label{eq:proof:est}
        \norm{\mathrm{T}_{\dsuper} f - \mathrm{T}_{U} f}{\Lt(\T^d)}^2 = \sum_{\u \in U_{\dsuper} \setminus U} \sum_{\k \in \Fud} \abs{\fc{\k}{f}}^2 \leq \va{f} \abs{U_{\dsuper} \setminus U} \varepsilon.
    \end{equation}
    Clearly, we have $\va{f} \leq \norm{f}{\Lt(\T^d)}^2 \leq \norm{f}{\st(\T^d)}^2$.
\end{proof}

\begin{theorem}\label{thm:av_trunc_error:2}
    Let $f \in \aw(\T^d)$ with a weight function $\fun{w}{\Z^d}{[1,\infty)}$. If there exsists a subset of ANOVA terms $U \subset U_{d_s}$, $d_s \in \D$, such that \begin{equation}\label{eq:linf:preq}
    	\frac{\sum_{\k\in\Fud} \abs{\fc{\k}{f}}}{\sum_{\k\in\Z^d} \abs{\fc{\k}{f}}}< \varepsilon_1,\,\varepsilon_1 >0 
    \end{equation} for every $\u \in U_{d_s} \setminus U$ and we have $$ \frac{\norm{f-\mathrm{T}_{{d_s}} f}{\mathrm{L}_\infty(\T^d)}}{\norm{f}{\aw(\T^d)}} < \varepsilon_2, \,\varepsilon_2 > 0,   $$ then \begin{equation*}
        \frac{\norm{f - \mathrm{T}_{U} f}{\mathrm{L}_\infty(\T^d)}}{\norm{f}{\aw(\T^d)}} \leq \varepsilon_2 + \sqrt{\abs{U_{d_s} \setminus U} \varepsilon_1}.
    \end{equation*}
\end{theorem}
\begin{proof}
    We split the ANOVA truncation error as in \eqref{eq:anv:trunc:err} and prove an upper bound for the second part. To this end, we estimate the $\mathrm{L}_\infty$ norm of $f$ by the absolute values of its Fourier coefficients and apply \eqref{eq:linf:preq} to obtain
    \begin{equation*}
        \norm{\mathrm{T}_{d_s} f - \mathrm{T}_{U} f}{\mathrm{L}_\infty(\T^d)} \leq \sum_{\u \in U_{d_s} \setminus U} \sum_{\k \in \Fud} \abs{\fc{\k}{f}} \leq \abs{U_{d_s} \setminus U} \varepsilon_1 \sum_{\k\in\Z^d} \abs{\fc{\k}{f}}.
    \end{equation*}
    Naturally, it holds that $\sum_{\k\in\Z^d} \abs{\fc{\k}{f}}\leq\norm{f}{\aw(\T^d)}$ which leads to the desired estimate.
\end{proof}
Note that in order to prove a bound for the error in $\mathrm{L}_\infty$, we formulated a condition on an $\ell_1$ equivalent of the global sensitivity indices $\gsi{\u}{f}$ in accordance with the Wiener algebra norm.

\subsection{Approximation error}\label{sec:error:approxerror}

In this section, we focus on the approximation error which we separate into two parts as well \begin{equation}\label{eq:approx_error}
    \norm{\mathrm{T}_{U} f - S_{I(U)}^X f}{H} \leq \underbrace{\norm{\mathrm{T}_{U} f - S_{I(U)} f}{H}}_{\text{truncation error}} + \underbrace{\norm{S_{I(U)} f - S_{I(U)}^X f}{H}}_{\text{aliasing error}}
\end{equation} with $H \in \{\Lt(\T^d), \L_\infty(\T^d)\}$, a subset of ANOVA terms $U \subset \mathcal{P}(\D)$, and a finite frequency index set $I(U) \subset \Z^d$ of structure \eqref{set:IUfull} with sets $I_{\u}$ as in \eqref{eq:approx:freq_weight}. The truncation error remains independent of the approximation scenario and can be estimated by the norms in $\aw$ and $\st$. \begin{lemma}\label{lem:error:6}
     Let $f \in \st(\T^d)$, $\fun{w}{\Z^d}{[1,\infty)}$ a weight function, and $I(U) \subset \Z^d$ a finite frequency index set of type \eqref{set:IUfull} with $U \subset \mathcal{P}(\D)$. Then the relative truncation error can be estimated as \begin{equation}\label{eq:error:trunc:l2hw}
        \frac{\norm{\mathrm{T}_{U} f - S_{I(U)} f}{\Lt(\T^d)}}{\norm{f}{\st(\T^d)}} \leq \frac{1}{\min_{\u\in U} N_{\u}}. 
    \end{equation}
    If in addition we have $\sum_{\k \in \Z^d} \frac{1}{w^2(\k)} < \infty$, we can estimate \begin{equation}\label{eq:error:trunc:lihw}
        \frac{\norm{\mathrm{T}_{U} f - S_{I(U)} f}{\L_\infty(\T^d)}}{\norm{f}{\st(\T^d)}} \leq \sqrt{\sum_{\u \in U} \sum_{\k \in \Fud\setminus \mathrm{P}_{\u}I_{\u}} \frac{1}{w^2(\k)}}. 
    \end{equation}
\end{lemma}
\begin{proof}
    In order to prove \eqref{eq:error:trunc:l2hw} we employ Parseval's identity and use the weight $w(\k)$ \begin{align*}
        \norm{\mathrm{T}_{U} f - S_{I(U)} f}{\Lt(\T^d)}^2 &= \sum_{\u \in U} \sum_{\k \in \Fud \setminus \mathrm{P}_{\u}I_{\u}} \abs{\fc{\k}{f}}^2 = \sum_{\u \in U} \sum_{\k \in \Fud \setminus \mathrm{P}_{\u}I_{\u}} \frac{w^2(\k)}{w^2(\k)}\abs{\fc{\k}{f}}^2 \\ &\!\!\leq \sum_{\u \in U} \frac{1}{N_{\u}^2} \sum_{\k \in \Fud \setminus \mathrm{P}_{\u}I_{\u}} w^2(\k) \abs{\fc{\k}{f}}^2 \leq \frac{1}{\min_{\u\in U} N_{\u}^2} \norm{f}{\st(\T^d)}^2.
    \end{align*}
    For the bound \eqref{eq:error:trunc:lihw} we estimate the norm by the absolute sum of the Fourier coefficients and use the Cauchy-Schwarz inequality \begin{align*}
        \norm{\mathrm{T}_{U} f - S_{I(U)} f}{\L_\infty(\T^d)} &= \sum_{\k \in \bigcup_{\u\in U} \Fud \setminus I(U)} \abs{\fc{\k}{f}} = \sum_{\k \in \bigcup_{\u\in U} \Fud \setminus I(U)} \frac{w(\k)}{w(\k)}\abs{\fc{\k}{f}} \\ &\leq \norm{f}{\st(\T^d)}\sqrt{\sum_{\u \in U} \sum_{\k \in \Fud\setminus \mathrm{P}_{\u}I_{\u}} \frac{1}{w^2(\k)}}.
    \end{align*}
\end{proof}
\begin{lemma}
     Let $f \in \aw(\T^d)$ with $\fun{w}{\Z^d}{[1,\infty)}$ a weight function such that $\sum_{\k \in \Z^d} \frac{1}{w^2(\k)} < \infty$, and $I(U) \subset \Z^d$ a finite frequency index set of type \eqref{set:IUfull} with $U \subset \mathcal{P}(\D)$ and sets $I_{\u}$ as in \eqref{eq:approx:freq_weight}. Then the relative truncation error can be estimated as \begin{equation*}
        \frac{\norm{\mathrm{T}_{U} f - S_{I(U)} f}{\L_\infty(\T^d)}}{\norm{f}{\aw(\T^d)}} \leq \min \left\{ \frac{1}{\min_{\u\in U} N_{\u}}, \max_{\u \in U} \sqrt{ \sum_{\k \in \Fud\setminus \mathrm{P}_{\u}I_{\u}}\frac{1}{w^2(\k)}} \right\}. 
    \end{equation*}
\end{lemma}
\begin{proof}
	The proof requires similar steps to the proof of \cref{lem:error:6}. 
\end{proof}

For the aliasing error in \eqref{eq:approx_error}, we start by considering the black-box approximation case where we solve the least-squares problem as described in \cref{sec:approx:ls}. 
\begin{theorem}\label{thm:anova:error:bb}                                                                                                 
    Let $f \in \st(\T^d)$ with a weight function $\fun{w}{\Z^d}{[1,\infty)}$ such that $\sum_{\k \in \Z^d} \frac{1}{w^2(\k)} < \infty$ and $I(U) \subset \Z^d$ a finite frequency index set of type \eqref{set:IUfull} with sets $I_{\u}$ as in \eqref{eq:approx:freq_weight}. Moreover, we have a reconstructing rank-1 lattice $\Lambda(\b z, M, I(U))$ for a generating vector $\b z \in \Z^d$ and lattice size $M \in \N$. Then the aliasing error can be estimated as \begin{equation}\label{eq:error:al:l2a2}
        \frac{\norm{S_{I(U)} f - S_{I(U)}^{\Lambda(\b z, M, I(U))} f}{\Lt(\T^d)}}{\norm{f}{\st(\T^d)}} \leq \sqrt{\sum_{\k \in \Z^d\setminus I(U)} \frac{1}{w^2(\k)}}.
    \end{equation} Furthermore, if $f \in \aw(\T^d)$ we get for the $\L_\infty$-norm \begin{equation}\label{eq:error:al:lia2}
        \frac{\norm{S_{I(U)} f - S_{I(U)}^{\Lambda(\b z, M, I(U))} f}{\L_\infty(\T^d)}}{\norm{f}{\aw(\T^d)}} \leq  \frac{1}{\min_{\u\in U} N_{\u}}.
    \end{equation}
\end{theorem}
\begin{proof}
    We show the bound \eqref{eq:error:al:l2a2} by first applying Parseval's identity and \eqref{eq:pre:error} \begin{align*}
        \norm{S_{I(U)} f - S_{I(U)}^{\Lambda(\b z, M, I(U))} f}{\Lt(\T^d)}^2 &= \sum_{\k \in I(U)} \abs{\hat{f}_{\k} - \fc{\k}{f}}^2 \\ 
        &= \sum_{\k \in I(U)} \abs{\sum_{\b h \in \Lambda^\bot(\b{z}, M)\setminus\{\b 0\}} \fc{\k + \b h}{f}}^2.
    \end{align*} We then incorporate the weight and utilize the Cauchy-Schwarz inequality to obtain \begin{align*}
	    \norm{S_{I(U)} f - S_{I(U)}^{\Lambda(\b z, M, I(U))} f}{\Lt(\T^d)}^2 &= \sum_{\k \in I(U)} \abs{\sum_{\b h \in \Lambda^\bot(\b{z}, M)\setminus\{\b 0\}} \frac{w(\k+\h)}{w(\k+\h)}\fc{\k + \b h}{f}}^2 \\ 
	    &\leq \sum_{\k \in I(U)} \left( \sum_{\b h \in \Lambda^\bot(\b{z}, M)\setminus\{\b 0\}} w^2(\k+\h)\abs{\fc{\k + \b h}{f}}^2 \right) \times \\
	    &\quad\quad\times \left( \sum_{\b h \in \Lambda^\bot(\b{z}, M)\setminus\{\b 0\}} \frac{1}{w^2(\k+\b h)} \right) \\
	\end{align*}
	From \cite[Lemma 8.13]{PlPoStTa18} we know that for fixed $\k \in I(U)$ we have disjoint sets $$M_{\k} \coloneqq \left\{ \k + \h \colon \h \in \Lambda^\bot(\b{z}, M)\setminus\{\b 0\} \right\} \subset \Z^d \setminus I(U).$$ This means we are able to estimate \begin{align*}
		\sum_{\b h \in \Lambda^\bot(\b{z}, M)\setminus\{\b 0\}} w^2(\k+\h)\abs{\fc{\k + \b h}{f}}^2 &= \sum_{\b\ell \in M_{\k}} w^2(\b\ell)\abs{\fc{\b\ell}{f}}^2 \\
		&\leq \sum_{\b\ell \in \Z^d} w^2(\b\ell)\abs{\fc{\b\ell}{f}}^2 = \norm{f}{\st(\T^d)}^2
	\end{align*} such that $$\norm{S_{I(U)} f - S_{I(U)}^{\Lambda(\b z, M, I(U))} f}{\Lt(\T^d)}^2 \leq \norm{f}{\st(\T^d)}^2 \sum_{\k \in I(U)} \sum_{\b h \in \Lambda^\bot(\b{z}, M)\setminus\{\b 0\}} \frac{1}{w^2(\k+\h)}.$$ Using that the sets $M_{\k}$ are disjoint and $\bigcup_{\k\in I(U)}M_{\k} \subset \Z^d\setminus I(U)$ yields
 	\begin{align*}
		\sum_{\k \in I(U)} \sum_{\b h \in \Lambda^\bot(\b{z}, M)\setminus\{\b 0\}} \frac{1}{w^2(\k+\h)} &= \sum_{\k \in I(U)} \sum_{\b\ell \in M_{\k}} \frac{1}{w^2(\b\ell)} \\  
		&= \sum_{\b\ell \in \bigcup_{\k \in I(U)} M_{\k}} \frac{1}{w^2(\b\ell)} \leq \sum_{\b\ell \in \Z^d \setminus I(U)} \frac{1}{w^2(\b\ell)}.
	\end{align*}
 	The $\L_\infty$-bound \eqref{eq:error:al:lia2} is obtained similarly to the method used in the proof of \cite[Theorem 8.14]{PlPoStTa18}. We proceed as follows \begin{align*}
    	\norm{S_{I(U)} f - S_{I(U)}^{\Lambda(\b z, M, I(U))} f}{\L_\infty(\T^d)} &\leq \sum_{\u \in U} \sum_{\k \in \Fud \setminus \mathrm{P}_{\u}I_{\u}} \abs{\fc{\k}{f}} \\
    	&= \sum_{\u \in U} \sum_{\k \in \Fud \setminus \mathrm{P}_{\u}I_{\u}} \frac{w(\k)}{w(\k)}\abs{\fc{\k}{f}} \\
    	&\leq \frac{1}{\min_{\u\in U} N_{\u}} \sum_{\u \in U} \sum_{\k \in \Fud \setminus \mathrm{P}_{\u}I_{\u}} w(\k) \abs{\fc{\k}{f}}.
    \end{align*}
	The result is obtained through estimating the sum by $\norm{f}{\aw(\T^d)}$.
\end{proof}

In the following we consider the approximation error for scattered data approximation with a fixed node set $X \subset \T^d$. Previously, we assumed that the index set $I(U)$ and the node set $X$ are such that the Fourier matrix $\F_{I(U)}$ has full rank. In this case the least-squares problem \eqref{eq:approx:gen_ls} has a unique solution. Assuming that the nodes in $X$ are i.i.d.~random variables that are uniformly distributed in $\T^d$, it is possible to achieve good bounds on the approximation error, see \cite{BaGr03, GrPoRa07, KaeUlVo19, MoeUl20}. 

\begin{lemma}\label{lem:error:l2est}
    Let $f \in \st(\T^d)$ with a weight function $\fun{w}{\Z^d}{[1,\infty)}$ such that $\sum_{\k \in \Z^d} \frac{1}{w^2(\k)} < \infty$, $X \subset \T^d$ a finite set of i.i.d.~uniformly distributed points, $\b y = (f(\x))_{\x\in X}$, and $I(U) \subset \Z^d$ a finite frequency index set of type \eqref{set:IUfull} with $U \subset \mathcal{P}(\D)$ a subset of ANOVA terms and sets $I_{\u}$ as in \eqref{eq:approx:freq_weight}. If for the number of frequencies we have $\abs{I(U)} \leq \frac{\abs{X}}{7r\log\abs{X}}, r > 0$, then \begin{equation*}
        \sup_{\norm{f}{\st(\T^d)} \leq 1} \frac{\sum_{\x\in X} \abs{\left(f-S_{I(U)} f\right)(\x)}^2}{\abs{X}} \leq 5 \max\left( \theta_{I(U)}^2,  \frac{8r\kappa^2\log\abs{X}}{\abs{X}}\!\!\! \sum_{\k\in\Z^d\setminus I(U)} \frac{1}{w^2(\k)} \right)
    \end{equation*} with a probability of at least $1-3\abs{X}^{1-r}$ for $\theta_{I(U)} = \norm{f - S_{I(U)}f}{\L_2(\T^d)}$ and $\kappa = \frac{1+\sqrt{5}}{2}$.
\end{lemma}
\begin{proof}
    The setting of this lemma is a special case of \cite[Theorem 5.1]{MoeUl20}.
\end{proof}
The following theorem deals with the actual approximation error by incorporating the previous lemma.
\begin{theorem}\label{thm:error:approx}
    Let $f \in \st(\T^d)$ with a weight function $\fun{w}{\Z^d}{[1,\infty)}$ such that $\sum_{\k \in \Z^d} \frac{1}{w^2(\k)} < \infty$, $I(U) \subset \Z^d$ a finite frequency index set of type \eqref{set:IUfull} with sets $I_{\u}$ as in \eqref{eq:approx:freq_weight}. Moreover, $U \subset \mathcal{P}(\D)$, and $S_{I(U)}^X f$ are the corresponding approximate Fourier partial sum obtained through the scattered data approximation method described in \cref{sec:approx:ls}. If the elements of $X \subset \T^d$ are i.i.d.~random variables uniformly distributed on $\T^d$ and for the number of frequencies we have $\abs{I(U)} \leq \frac{\abs{X}}{7r\log\abs{X}}, r > 0$, then
    \begin{equation*}
        \frac{\norm{S_{I(U)} f - S_{I(U)}^X f}{\L_2(\T^d)}}{\norm{f}{\st(\T^d)}} \leq \sqrt{ 8 \max\left( \theta_{I(U)}^2, \kappa^2 \frac{\log\abs{X}}{\abs{X}} \sum_{\k\in\Z^d\setminus I(U)} \frac{1}{w^2(\k)} \right) }
    \end{equation*} with a probability of at least $1-3\abs{X}^{1-r}$ for $\theta_{I(U)} = \norm{f - S_{I(U)}f}{\L_2(\T^d)}$ and $\kappa = \frac{1+\sqrt{5}}{2}$.
\end{theorem}
\begin{proof}
    We denote the Fourier coefficients with $\hat{\b c} = (\fc{\k}{f})_{\k\in I(U)}$ and the approximate Fourier coefficients computed by \cref{alg} with $\hat{\b f} = (\hat{f}_{\k})_{\k\in I(U)}$. With Parseval's identity as well as the Moore-Penrose inverse we obtain \begin{align*}
        \norm{S_{I(U)} f - S_{I(U)}^X f}{\L_2(\T^d)} &= \sqrt{\sum_{\k \in I(U)} \abs{\hat{f}_{\k}-\fc{\k}{f}}^2} = \norm{\hat{\b f} - \hat{\b c}}{2} \\
        &= \norm{\left( \F_{I(U)}^\ast \F_{I(U)} \right)^{-1} \F_{I(U)}^\ast \b y - \hat{\b c}}{2} \\
        &= \norm{\left( \F_{I(U)}^\ast \F_{I(U)} \right)^{-1} \F_{I(U)}^\ast \left( \b y - \F_{I(U)} \hat{\b c} \right)}{2}.
    \end{align*}
    We use the properties of the spectral norm and estimate further \begin{align*}
        &\leq \norm{\left( \F_{I(U)}^\ast \F_{I(U)}  \right)^{-1}\F_{I(U)}^\ast}{2}  \norm{\b y - \F_{I(U)} \hat{\b c}}{2} \\
        &= \norm{\left( \F_{I(U)}^\ast \F_{I(U)}  \right)^{-1}\F_{I(U)}^\ast}{2} \sqrt{ \sum_{\x\in X} \abs{\left(f-S_{I(U)} f\right)(\x)}^2. }
    \end{align*} 
    Applying \cite[Theorem 2.3]{MoeUl20} yields \begin{equation*}
        \sup_{\norm{f}{\st(\T^d)} \leq 1} \norm{S_{I(U)} f - S_{I(U)}^X f}{\L_2(\T^d)} \leq \sup_{\norm{f}{\st(\T^d)} \leq 1} \sqrt{\frac{2}{\abs{X}} \sum_{\x\in X} \abs{\left(f-S_{I(U)} f\right)(\x)}^2 }.
    \end{equation*} Finally, we use \cref{lem:error:l2est} to obtain our bound \begin{equation*}
        \sup_{\norm{f}{\st(\T^d)} \leq 1} \norm{S_{I(U)} f - S_{I(U)}^X f}{\L_2(\T^d)} \leq \sqrt{ 8 \max\left( \theta_{I(U)}^2, \kappa^2 \frac{\log\abs{X}}{\abs{X}} \sum_{\k\in\Z^d\setminus I(U)} \frac{1}{w^2(\k)} \right) }
    \end{equation*} with a probability of at least $1-3\abs{X}^{1-r}$. 
\end{proof}

This concludes the consideration of the error of the presented method in both approximation scenarios. We were able to achieve bounds for $\Lt$ and $\L_{\infty}$ for functions in weighted Wiener algebras and Sobolev type spaces.

\section{Numerical Results}\label{sec:numeric}
\begin{figure}[ht]
	\centering
	\subfloat[B-spline $B_2$]{ \begin{tikzpicture}[scale=0.75]
			\begin{axis}[
				grid=major,
				xmin=-0.05,
				xmax=1.05,
				ymin=-0.04,
				ymax=1.77
				]
				\addplot[smooth, line width=2pt] coordinates {
					(0.0, 0.0)(0.01, 0.034641016151377546)(0.02, 0.06928203230275509)(0.03, 0.10392304845413262)(0.04, 0.13856406460551018)(0.05, 0.17320508075688773)(0.06, 0.20784609690826525)(0.07, 0.24248711305964282)(0.08, 0.27712812921102037)(0.09, 0.31176914536239786)(0.1, 0.34641016151377546)(0.11, 0.381051177665153)(0.12, 0.4156921938165305)(0.13, 0.4503332099679081)(0.14, 0.48497422611928565)(0.15, 0.5196152422706631)(0.16, 0.5542562584220407)(0.17, 0.5888972745734183)(0.18, 0.6235382907247957)(0.19, 0.6581793068761733)(0.2, 0.6928203230275509)(0.21, 0.7274613391789284)(0.22, 0.762102355330306)(0.23, 0.7967433714816835)(0.24, 0.831384387633061)(0.25, 0.8660254037844386)(0.26, 0.9006664199358162)(0.27, 0.9353074360871938)(0.28, 0.9699484522385713)(0.29, 1.0045894683899488)(0.3, 1.0392304845413263)(0.31, 1.0738715006927038)(0.32, 1.1085125168440815)(0.33, 1.143153532995459)(0.34, 1.1777945491468367)(0.35, 1.212435565298214)(0.36, 1.2470765814495914)(0.37, 1.2817175976009691)(0.38, 1.3163586137523466)(0.39, 1.3509996299037244)(0.4, 1.3856406460551018)(0.41, 1.4202816622064791)(0.42, 1.4549226783578568)(0.43, 1.4895636945092343)(0.44, 1.524204710660612)(0.45, 1.5588457268119895)(0.46, 1.593486742963367)(0.47, 1.6281277591147445)(0.48, 1.662768775266122)(0.49, 1.6974097914174997)(0.5, 1.7320508075688772)(0.51, 1.6974097914174997)(0.52, 1.662768775266122)(0.53, 1.6281277591147445)(0.54, 1.5934867429633668)(0.55, 1.5588457268119893)(0.56, 1.5242047106606118)(0.57, 1.4895636945092345)(0.58, 1.454922678357857)(0.59, 1.4202816622064793)(0.6, 1.3856406460551018)(0.61, 1.3509996299037244)(0.62, 1.3163586137523466)(0.63, 1.2817175976009691)(0.64, 1.2470765814495914)(0.65, 1.212435565298214)(0.66, 1.1777945491468365)(0.67, 1.1431535329954587)(0.68, 1.1085125168440813)(0.69, 1.073871500692704)(0.7, 1.0392304845413265)(0.71, 1.004589468389949)(0.72, 0.9699484522385713)(0.73, 0.9353074360871938)(0.74, 0.9006664199358162)(0.75, 0.8660254037844386)(0.76, 0.831384387633061)(0.77, 0.7967433714816834)(0.78, 0.7621023553303059)(0.79, 0.7274613391789283)(0.8, 0.6928203230275507)(0.81, 0.6581793068761731)(0.82, 0.6235382907247959)(0.83, 0.5888972745734183)(0.84, 0.5542562584220408)(0.85, 0.5196152422706632)(0.86, 0.48497422611928565)(0.87, 0.4503332099679081)(0.88, 0.4156921938165305)(0.89, 0.38105117766515295)(0.9, 0.34641016151377535)(0.91, 0.3117691453623978)(0.92, 0.2771281292110202)(0.93, 0.24248711305964263)(0.94, 0.20784609690826544)(0.95, 0.17320508075688787)(0.96, 0.1385640646055103)(0.97, 0.10392304845413272)(0.98, 0.06928203230275515)(0.99, 0.034641016151377574)(1.0, 0.0)
				};		
			\end{axis}
	\end{tikzpicture} }	\hfill 	\subfloat[B-spline $B_4$]{ \begin{tikzpicture}[scale=0.75]
			\begin{axis}[
				grid=major,
				xmin=-0.05,
				xmax=1.05,
				ymin=-0.04,
				ymax=1.96
				]
				\addplot[smooth, line width=2pt] coordinates {
					(0.0, 0.0)(0.01, 3.081239967139746e-5)(0.02, 0.0002464991973711797)(0.03, 0.0008319347911277312)(0.04, 0.0019719935789694375)(0.05, 0.003851549958924683)(0.06, 0.00665547832902185)(0.07, 0.01056865308728933)(0.08, 0.0157759486317555)(0.09, 0.022462239360448742)(0.1, 0.030812399671397463)(0.11, 0.04101130396263001)(0.12, 0.0532438266321748)(0.13, 0.06769484207806022)(0.14, 0.08454922469831463)(0.15, 0.10399184889096641)(0.16, 0.126207589054044)(0.17, 0.15138131958557574)(0.18, 0.17969791488358994)(0.19, 0.21134224934611515)(0.2, 0.2464991973711797)(0.21, 0.2853536333568118)(0.22, 0.3280904317010401)(0.23, 0.37489446680189287)(0.24, 0.4259506130573984)(0.25, 0.4814437448655848)(0.26, 0.5414354870257959)(0.27, 0.6054944659426309)(0.28, 0.6730660584220054)(0.29, 0.7435956412698345)(0.3, 0.816528591292032)(0.31, 0.8913102852945141)(0.32, 0.9673861000831938)(0.33, 1.0442014124639882)(0.34, 1.1212015992428104)(0.35, 1.1978320372255755)(0.36, 1.2735381032181992)(0.37, 1.3477651740265955)(0.38, 1.4199586264566795)(0.39, 1.4895638373143671)(0.4, 1.556026183405572)(0.41, 1.6187910415362072)(0.42, 1.6773037885121904)(0.43, 1.731009801139436)(0.44, 1.7793544562238595)(0.45, 1.8217831305713739)(0.46, 1.8577412009878949)(0.47, 1.8866740442793377)(0.48, 1.9080270372516144)(0.49, 1.921245556710643)(0.5, 1.9257749794623442)(0.51, 1.9212455567106443)(0.52, 1.908027037251617)(0.53, 1.8866740442793455)(0.54, 1.8577412009879)(0.55, 1.8217831305713739)(0.56, 1.779354456223866)(0.57, 1.731009801139436)(0.58, 1.6773037885121955)(0.59, 1.6187910415362097)(0.6, 1.556026183405577)(0.61, 1.4895638373143678)(0.62, 1.4199586264566832)(0.63, 1.3477651740265832)(0.64, 1.2735381032182074)(0.65, 1.1978320372255775)(0.66, 1.1212015992428148)(0.67, 1.0442014124639747)(0.68, 0.9673861000831989)(0.69, 0.891310285294522)(0.7, 0.81652859129204)(0.71, 0.7435956412698232)(0.72, 0.6730660584219932)(0.73, 0.6054944659426306)(0.74, 0.5414354870257954)(0.75, 0.4814437448655835)(0.76, 0.4259506130573959)(0.77, 0.37489446680188887)(0.78, 0.3280904317010415)(0.79, 0.2853536333568124)(0.8, 0.24649919737118325)(0.81, 0.21134224934611526)(0.82, 0.1796979148835875)(0.83, 0.15138131958557657)(0.84, 0.1262075890540436)(0.85, 0.10399184889096778)(0.86, 0.08454922469831536)(0.87, 0.06769484207806548)(0.88, 0.053243826632176736)(0.89, 0.041011303962633396)(0.9, 0.030812399671401758)(0.91, 0.02246223936045325)(0.92, 0.0157759486317593)(0.93, 0.01056865308729904)(0.94, 0.0066554783290182435)(0.95, 0.0038515499589242574)(0.96, 0.001971993578967988)(0.97, 0.0008319347911208663)(0.98, 0.00024649919736458443)(0.99, 3.0812399665441756e-5)(1.0, 0.0)
				};
			\end{axis}
	\end{tikzpicture} } \\ 	\centering 	\subfloat[B-spline $B_6$]{ \begin{tikzpicture}[scale=0.75]
			\begin{axis}[
				grid=major,
				xmin=-0.05,
				xmax=1.05,
				ymin=-0.04,
				ymax=2.18
				]
				\addplot[smooth, line width=2pt] coordinates {	
					(0.0, 0.0)(0.01, 2.5289692437520977e-8)(0.02, 8.092701580006713e-7)(0.03, 6.1453952623175955e-6)(0.04, 2.589664505602148e-5)(0.05, 7.903028886725306e-5)(0.06, 0.00019665264839416305)(0.07, 0.0004250438607974152)(0.08, 0.0008286926417926874)(0.09, 0.0014933310487431756)(0.1, 0.0025289692437520978)(0.11, 0.00407293025675519)(0.12, 0.006292884748613218)(0.13, 0.009389885774204475)(0.14, 0.013601403545517286)(0.15, 0.019204360194742483)(0.16, 0.026518164537365996)(0.17, 0.03590774621082436)(0.18, 0.04778595413644685)(0.19, 0.06260928724471007)(0.2, 0.08086457211503663)(0.21, 0.10305380914231049)(0.22, 0.12967899872141667)(0.23, 0.16122696743177728)(0.24, 0.19815419422189037)(0.25, 0.24087163659386523)(0.26, 0.2897295567879627)(0.27, 0.3450023479671257)(0.28, 0.406873360401532)(0.29, 0.47541972765311097)(0.3, 0.550597192760102)(0.31, 0.6322249344215706)(0.32, 0.7199703931819678)(0.33, 0.8133340976156558)(0.34, 0.9116345404663899)(0.35, 1.013997633469989)(0.36, 1.1193687948907696)(0.37, 1.226541706399451)(0.38, 1.3341886591429843)(0.39, 1.4408909013756848)(0.4, 1.5451689860900217)(0.41, 1.6455131186473508)(0.42, 1.7404135044092317)(0.43, 1.8283906963681402)(0.44, 1.9080259427780824)(0.45, 1.9779915347863877)(0.46, 2.037081154063482)(0.47, 2.0842402204350563)(0.48, 2.1185962395117395)(0.49, 2.139489150321113)(0.5, 2.1465016729378554)(0.51, 2.1394891503211317)(0.52, 2.11859623951186)(0.53, 2.084240220435121)(0.54, 2.037081154063593)(0.55, 1.9779915347860548)(0.56, 1.9080259427783872)(0.57, 1.8283906963680663)(0.58, 1.740413504409583)(0.59, 1.6455131186473277)(0.6, 1.5451689860903821)(0.61, 1.4408909013759574)(0.62, 1.3341886591429288)(0.63, 1.226541706399176)(0.64, 1.119368794890587)(0.65, 1.0139976334700258)(0.66, 0.9116345404668544)(0.67, 0.8133340976153929)(0.68, 0.7199703931822653)(0.69, 0.6322249344218501)(0.7, 0.5505971927599074)(0.71, 0.47541972765268176)(0.72, 0.4068733604018832)(0.73, 0.3450023479675226)(0.74, 0.28972955678803114)(0.75, 0.24087163659398367)(0.76, 0.19815419422174782)(0.77, 0.16122696743193962)(0.78, 0.12967899872159389)(0.79, 0.10305380914261986)(0.8, 0.08086457211511765)(0.81, 0.06260928724453993)(0.82, 0.04778595413559988)(0.83, 0.03590774621046817)(0.84, 0.02651816453729843)(0.85, 0.019204360194451865)(0.86, 0.013601403545429673)(0.87, 0.009389885774181332)(0.88, 0.006292884748490307)(0.89, 0.004072930256630451)(0.9, 0.0025289692435232496)(0.91, 0.001493331048762266)(0.92, 0.0008286926417351147)(0.93, 0.000425043860960483)(0.94, 0.00019665264861852154)(0.95, 7.903028896351346e-5)(0.96, 2.5896644943300067e-5)(0.97, 6.145395484240691e-6)(0.98, 8.092701510118062e-7)(0.99, 2.52900550273672e-8)(1.0, 0.0)
				};
			\end{axis}
	\end{tikzpicture} }
	\caption{B-splines $B_2$, $B_4$, and $B_6$ over $\T \cong [0,1)$.}
	\label{fig:splines}
\end{figure}
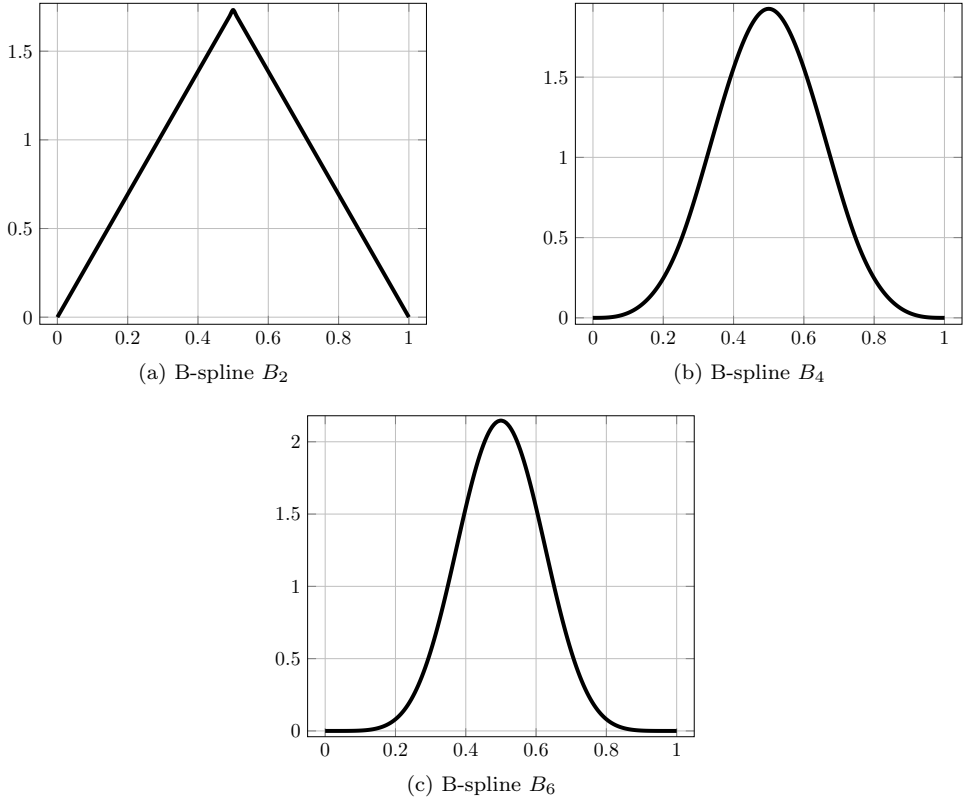

We present numerical results for the method described in \Cref{sec:anova:approxmethod} for a test function $\fun{f}{[0,1)^9}{\R}$, \begin{equation}\label{testfun}
	f(\x) \coloneqq B_2(x_1) B_4(x_5) + B_2(x_2) B_4(x_6) + B_2(x_3) B_4(x_7) + B_2(x_4) B_4(x_8) B_6(x_9) ,
\end{equation}
where $B_2$, $B_4$ and $B_6$ are parts of univariate, shifted, scaled and dilated B-splines of order 2, 4, and 6, respectively, see \Cref{fig:splines} for illustration. Their Fourier series is given by \begin{equation*}
    B_j(x) \coloneqq c_j \sum_{k \in \Z} \mathrm{sinc}^j\left(\frac{\pi\cdot k}{j}\right) \cos( \pi \cdot k ) \,\e^{2\pi\i k \cdot x}
\end{equation*} with $\mathrm{sinc}(x) \coloneqq \sin(x)/x$ and the constants $c_2 \coloneqq \sqrt{3/4}$, $c_4 \coloneqq \sqrt{315/604}$, $c_6 \coloneqq \sqrt{277200/655177}$ such that $\norm{B_j}{\L_2(\T^d)} = 1$. This allows the direct computation of the Fourier coefficients $\fc{\k}{f}$ and the norm $\norm{f}{\L_2(\T^d)}$. The ANOVA terms $f_{\u}$ are only nonzero for $$\u \in U^\ast \coloneqq \mathcal{P}(\{ 1,5 \}) \cup \mathcal{P}(\{ 2,6 \}) \cup \mathcal{P}(\{ 3,7 \})\cup \mathcal{P}(\{ 4,8,9 \}).$$ The function $f$ therefore has an exact low-dimensional structure for $d_s = 3$, i.e., $\mathrm{T}_{3} f = f$. This leads to $d_s = 3$ being the optimal choice for the superposition threshold with no error caused by ANOVA truncation since it corresponds to the superposition dimension $\dsuper$ for $\delta = 1$, see \eqref{eq:specificds}. In an approximation scenario with an unknown function $f$ this information is of course not known.

We consider two errors \begin{equation}
\varepsilon_{\ell_2} = \frac{\norm{\b y - (S_{I(U_{d_s})}^X f (\x))_{\x\in X}}{2}}{\norm{\b y}{2}}, \text{and} \quad \varepsilon_{\mathrm{L}_2} = \frac{\norm{f - S_{I(U_{d_s})}^X f}{\L_2(\T^9)}}{\norm{f}{\L_2(\T^9)}}.
\end{equation} Here, the error $\varepsilon_{\ell_2}$ can be regarded as a training error since it is taken at the given sampling set $X$ and the error $\varepsilon_{\mathrm{L}_2}$ as a type of generalization error since it measures the error in the Fourier coefficients. Since our goal is to find the important ANOVA terms, i.e., the terms in $U^\ast$, we expect to have an interval (or gap) in which to choose the order-dependent threshold $\b\varepsilon \in [0,1]^{d_s}$. Therefore, we define \begin{equation*}
I^{(j)} = \begin{cases}
\emptyset \quad &\colon \text{assumption \eqref{assumption} is not fulfilled} \\
(a^{(j)}, b^{(j)}) &\colon \text{assumption \eqref{assumption} is fulfilled}
\end{cases}
\end{equation*}
with $1 \leq j \leq d_s$ and \begin{align*}
a^{(j)} &\coloneqq \max \left\{ \gsi{\u}{S_{I(U_{d_s})}^X f} \colon \u \in U_{d_s} \setminus U^\ast, \au = j \right\}, \\
b^{(j)} &\coloneqq \min \left\{ \gsi{\u}{S_{I(U_{d_s})}^X f} \colon \u \in U^\ast, \au = j \right\}.
\end{align*} Here, the assumption \eqref{assumption} is to be understood for every order of terms, i.e., for $\u$ and $\v$ with $\au = \av = j$.

\begin{remark}
	The norm occurring in the error $\varepsilon_{\mathrm{L}_2}$ can be calculated using Parseval's identity \begin{equation*}
	\norm{f - S_{I(U_{d_s})}^X f}{\L_2}^2 = \norm{f}{\L_2}^2 + \sum_{\k \in I(U_{d_s})} \abs{ \fc{\k}{f} - \hat{f}_{\k} }^2  - \sum_{\k \in I(U_{d_s})} \abs{ \fc{\k}{f} }^2
	\end{equation*} which is possible since we know the exact Fourier coefficients and the norm of the function $f$. In general, this error cannot be computed.
\end{remark}

\subsection{Scattered Data Approximation}

For our numerical experiments we use one sampling set $X \subset\T^9$ of uniformly distributed nodes with $M \coloneqq \abs{X} = 2.5 \cdot 10^6$, and an evaluation vector $\b y = ( f(\x) )_{\x\in X}$. We are going to start by choosing three as the superposition threshold $d_s$ while later reducing it to two which allows us to see the effect of truncating an ANOVA term. Our primary aim for now is to detect the ANOVA terms in $U^\ast$ which we achieve using the first step of our method, see \cref{sec:approx:active}. To this end, we choose a frequency index set $I(U_{d_s}) \subset \Z^9$, cf.~\eqref{eq:approx:partial}, through order-dependent sets $I_0 = \{0\}$, $I_1 = \{-N_1/2,\dots,N_1/2-1\}$, $I_2 = \{-N_2/2,\dots,N_2/2-1\}^2$, and $I_3 = \{-N_3/2,\dots,N_3/2-1\}^3$ with $N_1,N_2,N_3 \in 2\N$. The method gives us an approximation $S_{I(U_{d_s})}^X f$.

Results of numerical experiments with the function $f$ from \eqref{testfun} and different choices for the bandwidths $N_1, N_2$, and $N_3$ are displayed in \Cref{tab:res_det}. They show that it is indeed possible to detect the ANOVA terms in $U^\ast$ using trigonometric polynomials of small degrees. Moreover, both errors are roughly of the same order. Since our number of samples $M$ is fixed, we are looking for values $\b N$ such that one balances the effects of underfitting and overfitting. The experiments suggest that the choice in examples 5 and 8 is close to optimal. In \Cref{fig:gsi} we depicted the global sensitivity indices $\gsi{\u}{S_{I(U_{d_s})}^X f}$, cf.~\cref{alg}, for example 8 from \Cref{tab:res_det}. The one-dimensional sets $\{i\}$, $i=1,\dots,9$, all have large indices as they are all in $U^\ast$ while the two dimensional sets $$\{1,5\}, \{2,6\}, \{3,7\}, \{4,8\}, \{4,9\}, \{8,9\} \in U^\ast$$ are clearly separated from the two dimensional sets in $U_{d_s}\setminus U^\ast$. The same holds for the one three-dimensional term $\{4,8,9\} \in U^\ast$. The size of the intervals $I^{(j)}$ suitable to choose the parameters $\varepsilon_j$ is especially relevant since it separates \textit{important} from \textit{unimportant} terms. 

\begin{table*}[ht]\centering
	\begin{tabular}{@{}rrrrrrrrrr@{}}\toprule& \multicolumn{2}{c}{size of index sets} & \phantom{a}& \multicolumn{2}{c}{relative errors} & \phantom{a} & \phantom{a}\\\cmidrule{2-3} \cmidrule{5-6} & $\b N$ & $\abs{I(U_{d_s})}$ && $\varepsilon_{\ell_2}$ & $\varepsilon_{\mathrm{L}_2}$  && $I^{(1)}, \, I^{(2)}, \, I^{(3)}$ \\\midrule 
		1 & [256, 32, 8] & 65704 && $ 4.7 \cdot {10}^{-3}$ & $4.8 \cdot {10}^{-3}$ && $(0.0, 0.021 )$ \\
		 &  &  &&  &  && $(3.0 \cdot {10}^{-8}, 0.019 )$ \\
		 &  &  && & && $(1.2 \cdot {10}^{-8}, 0.026 )$ \\ \bottomrule
		2 & [256, 32, 16] & 320392 && $ 2.1 \cdot {10}^{-3}$ & $2.4 \cdot {10}^{-3}$ && $(0.0, 0.021 )$ \\
		 &  &  &&  &  && $(7.2 \cdot {10}^{-9}, 0.019 )$ \\
		 &  &  && & && $(2.5 \cdot {10}^{-8}, 0.026 )$ \\ \bottomrule
		3 & [256, 32, 32] & 2539336 && $ 2.6 \cdot {10}^{-2}$ & $2.8 \cdot {10}^{-2}$ && $(0.0, 0.016 )$ \\
		&  &  &&  &  && $(8.3 \cdot {10}^{-5}, 0.015 )$ \\
		&  &  && & && $(2.5 \cdot {10}^{-3}, 0.023 )$ \\ \bottomrule
		4 & [256, 64, 8] & 173992 && $ 4.4 \cdot {10}^{-3}$ & $4.7 \cdot {10}^{-3}$ && $(0.0, 0.021 )$ \\
		 &  &  &&  &  && $(1.1 \cdot {10}^{-7}, 0.019 )$ \\
		 &  &  && & && $(1.1 \cdot {10}^{-8}, 0.026 )$ \\ \bottomrule
		5 & [256, 64, 16] & 428680 && $ 1.6 \cdot {10}^{-3}$ & $1.9 \cdot {10}^{-3}$ && $(0.0, 0.021 )$ \\
		 &  &  &&  &  && $(1.8 \cdot {10}^{-8}, 0.019 )$ \\
		 &  &  && & && $(1.6 \cdot {10}^{-8}, 0.026 )$ \\ \bottomrule
		6 & [256, 64, 32] & 2647624 && $ 2.5 \cdot {10}^{-2}$ & $3.2 \cdot {10}^{-2}$ && $(0.0, 0.015 )$ \\
		 &  &  &&  &  && $(4.0 \cdot {10}^{-4}, 0.015 )$ \\
		 &  &  && & && $(2.9 \cdot {10}^{-3}, 0.022 )$ \\ \bottomrule
		7 & [512, 64, 8] & 176296 && $ 4.4 \cdot {10}^{-3}$ & $4.7 \cdot {10}^{-3}$ && $(0.0, 0.021 )$ \\
		 &  &  &&  &  && $(1.1 \cdot {10}^{-7}, 0.019 )$ \\
		 &  &  && & && $(1.1 \cdot {10}^{-8}, 0.026 )$ \\ \bottomrule
		8 & [512, 64, 16] & 430984 && $ 1.6 \cdot {10}^{-3}$ & $1.9 \cdot {10}^{-3}$ && $(0.0, 0.021 )$ \\
		 &  &  &&  &  && $(1.8 \cdot {10}^{-8}, 0.019 )$ \\
		 &  &  && & && $(1.6 \cdot {10}^{-8}, 0.026 )$ \\ \bottomrule
		9 & [512, 64, 32] & 2649928 && $ 2.5 \cdot {10}^{-2}$ & $3.2 \cdot {10}^{-2}$ && $(0.0, 0.015 )$ \\
		 &  &  &&  &  && $(4.0 \cdot {10}^{-4}, 0.015 )$ \\
		 &  &  && & && $(2.9 \cdot {10}^{-3}, 0.022 )$ \\ \bottomrule
	\end{tabular}
	\caption{Results of detection step for important ANOVA terms with $M = 2.5 \cdot 10^6$ uniformly distributed nodes ($\b N = [N_1,N_2,N_3]$).}
	\label{tab:res_det}
\end{table*}

\begin{figure}[ht!]
	\centering
	\subfloat[1d sensitivity indices]{ \begin{tikzpicture}
		\begin{axis}[clip=false, enlargelimits=false,ymode = log, xmin=0.8,xmax=9.2,ymin=1.85e-2,ymax=0.145,height=0.33\textwidth, width=0.43\textwidth, grid=major, xlabel={ANOVA term $\u$}, ylabel={$\gsi{\b{u}_i}{S_{I(U_{d_s})}^{X} f}$},
		legend style={at={(0.5,1.07)}, anchor=south,legend columns=3,legend cell align=left, font=\small, 
		},]
		\addplot[orange,mark=o,mark size=1] coordinates {
			(1, 0.13485590547067322) (2, 0.1348546191375092) (3, 0.1348538986781466) (4, 0.08479925199524384) (5, 0.05705736651807279) (6, 0.048995887099158836) (7, 0.0489958331178174) (8, 0.04899466500463767) (9, 0.020729792280798152) 
		};
		\end{axis}
		\end{tikzpicture} }
		\hfill 
		\subfloat[2d sensitivity indices]{ \begin{tikzpicture}
			\begin{axis}[clip=false, enlargelimits=false,ymode = log, xmin=0.2,xmax=37.2,ymin=1e-8,ymax=0.15,height=0.33\textwidth, width=0.43\textwidth, grid=major, xlabel={ANOVA term $\u$}, ylabel={$\gsi{\b{u}_i}{S_{I(U_{d_s})}^{X} f}$},
			legend style={at={(0.5,1.07)}, anchor=south,legend columns=3,legend cell align=left, font=\small, 
			},]
			\addplot[cyan,mark=o,mark size=1] coordinates {
				(1, 0.07780131659625403) (2, 0.04495099140872069) (3, 0.044950575872343745) (4, 0.04494983891678797) (5, 0.028265903218229544) (6, 0.019018986643685766) (7, 1.7828408967506963e-8) (8, 1.7639235059540238e-8) (9, 1.753033672139588e-8) (10, 1.7509397974474565e-8) (11, 1.7507769622850096e-8) (12, 1.7449770371942553e-8) (13, 1.7403659880631815e-8) (14, 1.733987266977155e-8) (15, 1.733601510416636e-8) (16, 1.7265892101823815e-8) (17, 1.7246212135785187e-8) (18, 1.722877931198043e-8) (19, 1.7193818854818427e-8) (20, 1.7165417106311266e-8) (21, 1.7154092722010225e-8) (22, 1.7147041508873725e-8) (23, 1.7137614109025423e-8) (24, 1.7135074581718182e-8) (25, 1.7086232174521464e-8) (26, 1.705219816506131e-8) (27, 1.7023378050871612e-8) (28, 1.694772280744182e-8) (29, 1.689329990485209e-8) (30, 1.6875238820966038e-8) (31, 1.683418685825038e-8) (32, 1.681973566786908e-8) (33, 1.6725646959476134e-8) (34, 1.662394021302858e-8) (35, 1.638422297263428e-8) (36, 1.5990670530065094e-8) 
			};
			\end{axis}
		\end{tikzpicture} } \\ \centering \subfloat[3d sensitivity indices]{ \begin{tikzpicture}
					\begin{axis}[clip=false, enlargelimits=false,ymode = log, xmin=1e-32,xmax=86,ymin=1e-8,ymax=0.0425,height=0.33\textwidth, width=0.43\textwidth, grid=major, xlabel={ANOVA term $\u$}, ylabel={$\gsi{\b{u}_i}{S_{I(U_{d_s})}^{X} f}$},
					legend style={at={(0.5,1.07)}, anchor=south,legend columns=3,legend cell align=left, font=\small, 
					},]
					\addplot[red,mark=o,mark size=0.8] coordinates {
								(1, 0.025923436849895076) (2, 1.5636814887748134e-8) (3, 1.558107839211133e-8) (4, 1.5480025651304596e-8) (5, 1.5465493974679868e-8) (6, 1.5453397971798796e-8) (7, 1.5451313438918488e-8) (8, 1.5271801006642355e-8) (9, 1.5247069465951782e-8) (10, 1.5182298763809362e-8) (11, 1.5148193138965472e-8) (12, 1.5137924399351542e-8) (13, 1.5124650776543285e-8) (14, 1.5117131441947184e-8) (15, 1.5074741510945442e-8) (16, 1.5068698960667126e-8) (17, 1.5054858425082342e-8) (18, 1.5051536628690458e-8) (19, 1.5023424089303886e-8) (20, 1.5002598185329892e-8) (21, 1.4982342466932743e-8) (22, 1.497781365817611e-8) (23, 1.496904113659431e-8) (24, 1.495921059689189e-8) (25, 1.4947773872235423e-8) (26, 1.493329209812126e-8) (27, 1.4930594045393455e-8) (28, 1.4920379073640693e-8) (29, 1.4903193824956635e-8) (30, 1.4896502497880414e-8) (31, 1.4891253059682746e-8) (32, 1.4850414150613423e-8) (33, 1.4828179165846257e-8) (34, 1.4804223671623508e-8) (35, 1.4760210630713028e-8) (36, 1.4714919938846229e-8) (37, 1.4703848815528526e-8) (38, 1.4703327069782456e-8) (39, 1.4685218540187085e-8) (40, 1.4657971459276087e-8) (41, 1.4657259097788491e-8) (42, 1.464863999039876e-8) (43, 1.4647531122299409e-8) (44, 1.4647347002925791e-8) (45, 1.4629718588642276e-8) (46, 1.4601491043610848e-8) (47, 1.4600113746504334e-8) (48, 1.459537477070138e-8) (49, 1.4593718478164756e-8) (50, 1.4576223379491161e-8) (51, 1.4538205552061024e-8) (52, 1.4522648520685689e-8) (53, 1.4516573899884193e-8) (54, 1.450158174768511e-8) (55, 1.447982112783596e-8) (56, 1.44687478536574e-8) (57, 1.4442029047165095e-8) (58, 1.4421212974093971e-8) (59, 1.4415809530158984e-8) (60, 1.4394865971449321e-8) (61, 1.4386567571464522e-8) (62, 1.4385709603856557e-8) (63, 1.4381149561484287e-8) (64, 1.4363754569460072e-8) (65, 1.4337066554370508e-8) (66, 1.4322049987350518e-8) (67, 1.4282614613329381e-8) (68, 1.427421152021412e-8) (69, 1.4258434646786866e-8) (70, 1.4255342911994694e-8) (71, 1.4241827453895658e-8) (72, 1.4221117672232602e-8) (73, 1.4216656983774674e-8) (74, 1.4208659551181513e-8) (75, 1.4207075513391394e-8) (76, 1.4198279363719234e-8) (77, 1.4166787981457814e-8) (78, 1.4154199966976589e-8) (79, 1.4067387735682094e-8) (80, 1.399643820825167e-8) (81, 1.3950645751339707e-8) (82, 1.3866474238262689e-8) (83, 1.3832030693735835e-8) (84, 1.3650986242283137e-8)
							};
					\end{axis}
				\end{tikzpicture} }
	\caption{Behavior of the global sensitivity indices $\gsi{\b u}{S_{I(U_{d_s})}^{X} f}$ for the example 8 from Table \ref{tab:res_det}.} 
	\label{fig:gsi}
\end{figure}
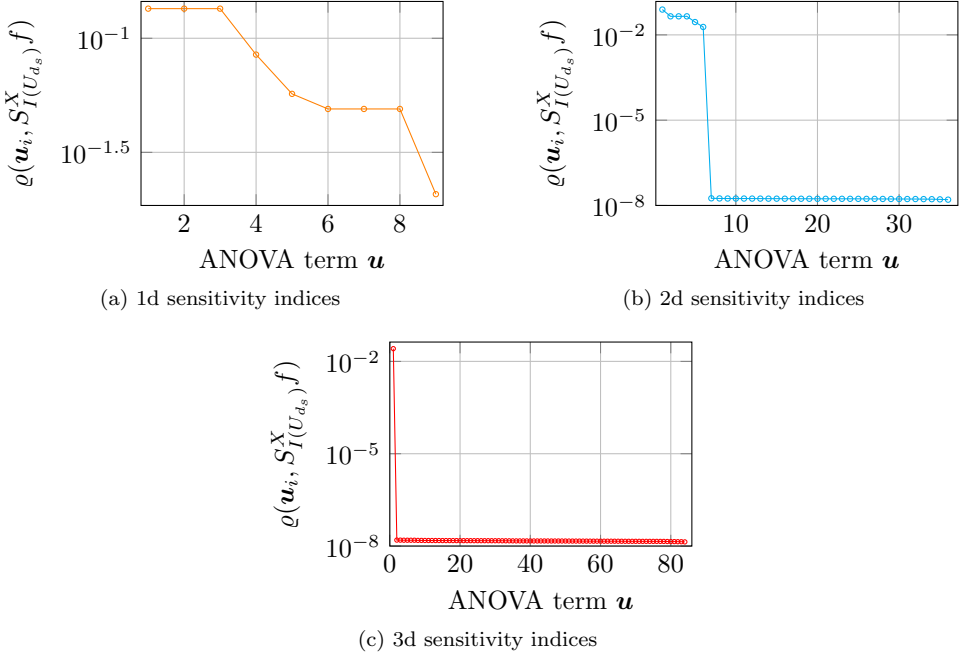

Since there exists $\b N$, and $\b\varepsilon$ such that we are able to recover the set of ANOVA terms $U^\ast$, we set $U_{X,\b y}^{(\b\varepsilon)} = U^\ast$ from now on. We aim to improve our approximation quality with the given data by solving the minimization problem \eqref{eq:approx:last}. Here, we could choose individual index sets for every ANOVA term in $U^\ast$ to form $I(U^\ast)$ based on the global sensitivity indices, but for our function order-dependence can be maintained. \Cref{tab:res_approx} shows the results of the approximation using the index set $I(U^\ast)$. 

\begin{table*}[ht]\centering
	\begin{tabular}{@{}rrrrrr@{}}\toprule& \multicolumn{2}{c}{size of index sets} & \phantom{a}& \multicolumn{2}{c}{relative errors} \\\cmidrule{2-3} \cmidrule{5-6} & $\b N$ & $\abs{I(U^\ast)}$ && $\varepsilon_{\ell_2}$ & $\varepsilon_{\mathrm{L}_2}$ \\\midrule 
		1 & [1024, 64, 64]  & 283069 && $ 5.6 \cdot {10}^{-4}$ & $ 6.3 \cdot {10}^{-4}$ \\
		2 & [1024, 128, 32]  & 135773  && $6.0 \cdot {10}^{-4}$ & $6.4  \cdot {10}^{-4}$ \\
		4 & [1024, 128, 64]  & 356029 && $2.7 \cdot {10}^{-4}$ & $3.1 \cdot {10}^{-4}$ \\
		5 & [1024, 256, 64] & 649405 && $2.0 \cdot {10}^{-4}$ & $2.7 \cdot {10}^{-4}$ \\ \bottomrule
	\end{tabular}
	\caption{Results for approximation with active set $U^\ast$ and $M = 2.5 \cdot 10^6$ uniformly distributed nodes ($\b N = [N_1,N_2,N_3]$).}
	\label{tab:res_approx}
\end{table*}

The number of terms in $U^\ast$ is significantly smaller than in $U_{d_s}$ such that we are able to increase $\b N$ while balancing the effects of over- and underfitting. We observe that the reduction of the ANOVA terms to $U^\ast$ yields benefit with regard to approximation quality due to the reduction in model complexity.

Now that we have experiments with no truncation error in the ANOVA decomposition, i.e., $\mathrm{T}_3 f = f$, we repeat the tests with a superposition threshold $d_s = 2$. In this case, it is not possible to detect the ANOVA term $f_{\{4,8,9\}}$ which results in the set $U^+ \coloneqq U^\ast \setminus \{4,8,9\}$ being optimal for the detection step. For the following tests, we use the same nodes as we did previously.

\begin{table*}[ht]\centering
	\begin{tabular}{@{}rrrrrrrrrr@{}}\toprule& \multicolumn{2}{c}{size of index sets} & \phantom{a}& \multicolumn{2}{c}{relative errors} & \phantom{a} & \phantom{a}\\\cmidrule{2-3} \cmidrule{5-6} & $\b N$ & $\abs{I(U_{d_s})}$ && $\varepsilon_{\ell_2}$ & $\varepsilon_{\mathrm{L}_2}$  && $I^{(1)}, \, I^{(2)}$  \\\midrule 
		1 & [256, 16] & 10396 && $9.4 \cdot {10}^{-2}$ & $9.4 \cdot {10}^{-2}$ && $(0.0,  0.021 )$ \\
		  &  &  && & && $(3.0 \cdot {10}^{-6},  0.020 )$ \\ \bottomrule
		2 & [256, 32] & 36892 && $9.3 \cdot {10}^{-2}$ & $9.4 \cdot {10}^{-2}$ && $( 0.0, 0.021 )$ \\
		  &  &  && & && $(3.0 \cdot {10}^{-6},  0.020 )$ \\ \bottomrule
		3 & [256, 64] & 145180 && $9.1 \cdot {10}^{-2}$ & $9.6 \cdot {10}^{-2}$ && $( 0.0,  0.021 )$ \\
		  &  &  && & && $(4.8 \cdot {10}^{-5},  0.020 )$ \\ \bottomrule
		4 & [256, 128] & 582940 && $8.2 \cdot {10}^{-2}$ & $1.1 \cdot {10}^{-1}$ && $( 0.0,  0.021 )$ \\
		  &  &  && & && $(2.3 \cdot {10}^{-4},  0.020 )$ \\ \bottomrule
	\end{tabular}
	\caption{Results of detection step for important ANOVA terms with $M = 2.5 \cdot 10^6$ uniformly distributed nodes and superposition threshold $d_s = 2$, ($\b N = [N_1,N_2]$).}
	\label{tab:res_det:ds2}
\end{table*}

The results of the experiments in \cref{tab:res_det:ds2} show that it is possible to determine the terms in $U^+$. Since three-dimensional terms are not included, the term $f_{\{4,8,9\}}$ is not in the approximation which results in the larger errors compared to \cref{tab:res_det}.

Since there exists $\b N \in \N^2$ and $\epsilon > 0$ such that $U_{X,\b y}^{(\b\varepsilon)} = U^+$, we use $U^+$ for the next approximation step with suitable index sets $I(U^+)$. The results for different choices of $N_1$ and $N_2$ are displayed in \cref{tab:res_det:ds2:res}. We are able to achieve better errors with the smaller index sets. Obviously, the influence of the cutoff error is dominating such that a large benefit in taking many additional frequencies cannot be observed.

\begin{table*}[ht]\centering
	\begin{tabular}{@{}rrrrrrrrrr@{}}\toprule& \multicolumn{2}{c}{size of index sets} & \phantom{a}& \multicolumn{2}{c}{relative errors} \\\cmidrule{2-3} \cmidrule{5-6} & $\b N$ & $\abs{I(U^+)}$ && $\varepsilon_{\ell_2}$ & $\varepsilon_{\mathrm{L}_2}$   \\\midrule 
		1 & [1024, 16] & 10558 && $9.3 \cdot {10}^{-2}$ & $9.3 \cdot {10}^{-2}$  \\
		2 & [1024, 32] & 14974 && $9.3 \cdot {10}^{-2}$ & $9.3 \cdot {10}^{-2}$  \\
		3 & [1024, 64] & 33022 && $9.3 \cdot {10}^{-2}$ & $9.4 \cdot {10}^{-2}$  \\
		4 & [1024, 128] & 105982 && $9.1 \cdot {10}^{-2}$ & $9.5 \cdot {10}^{-2}$  \\\bottomrule
	\end{tabular}
	\caption{Approximation results for active set $U^+$ with $M = 2.5 \cdot 10^6$ uniformly distributed nodes ($\b N = [N_1,N_2]$).}
	\label{tab:res_det:ds2:res}
\end{table*}

\subsection{Black-Box Approximation}

In the following numerical experiments we aim to find reconstructing rank-1 lattice, see \cref{sec:pre:lattice}, for the function $f$. In the first step, our goal is to determine the set of ANOVA terms $U^\ast$ and later use it to improve our approximation quality. As discussed in \cite{KaPoVo17}, the function $f$ works well with hyperbolic cross index sets of dominating mixed smoothness $3/2$. Therefore, we define \begin{equation}
	\mathcal{H}_j^{N} = \left\{ \k\in\Z^j \colon \prod_{s \in \supp \k} (1+\abs{k_s})^{\frac{3}{2}} \leq N \right\}, N \in \N.
\end{equation}
We choose as order-dependent index sets $I_0 = \{0\}$, $I_1 = \mathcal{H}_1^{N_1}$, $I_2 = \mathcal{H}_2^{N_2}$, and $I_3 = \mathcal{H}_3^{N_1}$ with $N_1,N_2,N_3 \in\N$ to obtain $I(U_{d_s})$ as in \eqref{eq:IU}. The method then gives us a reconstructing rank-1 lattice $X \coloneqq \Lambda(\b z, M, I(U_{d_s}))$ with generating vector $\b z \in \Z^9$ and lattice sizes $M \in \N$ by employing the component-by-component construction from \cite[Algorithm 8.17]{PlPoStTa18}. The approximation is defined as $S_{I(U_{d_s})}^X f$.

\Cref{tab:res:bb:3:full} shows results of numerical experiments with $f$, see \eqref{testfun}, and different choices for the parameters $N_1, N_2$, and $N_3$. We can see that there exist an $\b\varepsilon$ such that it is possible to detect the active set of terms $U^\ast$ in every test scenario. The lattice size increases with the growing index set as expected. Note that is sufficient to use an index set of 3481 frequencies and a lattice with only $46351$ evaluations in order to detect the active set of ANOVA terms.

Now, we set the active set $U_{X,\b y}^{(\b\varepsilon)} = U^\ast$. The aim is again to improve our approximation quality by solving the minimization problem \eqref{eq:approx:last}. We also maintain the order-dependence of the set $I(U^\ast)$ based on the structure of the function. \Cref{tab:res:bb:3:active} shows the results of the approximation using the index set $I(U^\ast)$. Larger cutoff parameters $N_i$ become possible such that we are able to achieve a good approximation error with relatively small lattice sizes in relation to our problem dimension. The sizes of our reconstructing lattices stay manageable as well.

\begin{table*}[ht]\centering
	\begin{tabular}{@{}rrrrrrrrrr@{}}\toprule \multicolumn{2}{c}{size of index sets} &  \multicolumn{2}{c}{relative errors} & \phantom{a} & \phantom{a}\\\cmidrule{1-4}   $\b N$ & $\abs{I(U_{d_s})}$ & $\varepsilon_{\ell_2}$ & $\varepsilon_{\mathrm{L}_2}$  &$M$& $I^{(1)}, \, I^{(2)}, \, I^{(3)}$  \\\midrule 
		 $[10^2, 10^2, 10^2]$ & 3481 & $2.8 \cdot {10}^{-2}$ & $3.0 \cdot {10}^{-2}$ & 47351 & $( 0.0, 0.021 )$ \\
		   &  &  & &  & $( 3.4 \cdot {10}^{-5}, 0.019 )$ \\
		   &  &  & &  & $( 5.7 \cdot {10}^{-5}, 0.025 )$ \\\bottomrule
		 $[10^3, 10^3, 10^3]$ & 11203 & $1.0 \cdot {10}^{-2}$ & $1.0 \cdot {10}^{-2}$ & 490277 & $( 0.0, 0.021 )$ \\
		  &  &  & &  & $( 1.7 \cdot {10}^{-7}, 0.019 )$ \\
		   &  &  & &  & $( 5.7 \cdot {10}^{-7}, 0.026 )$ \\\bottomrule
		 $[10^4, 10^4, 10^3]$ & 16891 & $7.0 \cdot {10}^{-3}$ & $7.1 \cdot {10}^{-3}$ & 1114489 & $( 0.0, 0.021 )$ \\
		  &  &  & &  & $( 3.5 \cdot {10}^{-10}, 0.019 )$ \\
		  &  &  & &  & $( 5.3 \cdot {10}^{-8}, 0.026 )$ \\\bottomrule
	 	$[10^5, 10^4, 10^3]$ & 17341 & $7.0 \cdot {10}^{-3}$ & $7.0 \cdot {10}^{-3}$ & 2349307 & $( 0.0, 0.021 )$ \\
		   &  &  & &  & $( 1.7 \cdot {10}^{-9}, 0.019 )$ \\
		   &  &  & &  & $( 2.8 \cdot {10}^{-9}, 0.026 )$ \\\bottomrule
	\end{tabular}
	\caption{Results of detection step for important ANOVA terms ($\b N = [N_1,N_2,N_3]$).}
	\label{tab:res:bb:3:full}
\end{table*}

\begin{table*}[ht]\centering
	\begin{tabular}{@{}rrrrrrrrrr@{}}\toprule \multicolumn{2}{c}{size of index sets} &  \multicolumn{2}{c}{relative errors} & \phantom{a} \\\cmidrule{1-4}   $\b N$ & $\abs{I(U_{d_s})}$ & $\varepsilon_{\ell_2}$ & $\varepsilon_{\mathrm{L}_2}$  &$M$\\\midrule 
		 $[10^4, 10^4, 10^4]$ & 2243 & $2.4 \cdot {10}^{-3}$ & $2.6 \cdot {10}^{-3}$ & 157243  \\
		 $[10^5, 10^5, 10^5]$ & 6565 & $8.3 \cdot {10}^{-4}$ & $8.3 \cdot {10}^{-4}$ & 1346881  \\
		 $[10^6, 10^5, 10^5]$ & 7591 & $7.7 \cdot {10}^{-4}$ & $7.7 \cdot {10}^{-4}$ & 883391 \\
	 	$[10^6, 10^6, 10^5]$ & 13495 & $5.0 \cdot {10}^{-4}$ & $5.0 \cdot {10}^{-4}$ & 5691109  \\ \bottomrule
	\end{tabular}
	\caption{Results of detection step for important ANOVA terms ($\b N = [N_1,N_2,N_3]$).}
	\label{tab:res:bb:3:active}
\end{table*}

\section{Summary}

In this paper we considered the classical ANOVA decomposition for periodic functions. We studied different index sets $\Pud$ and $\Fud$ for the projections $\mathrm{P}_{\u} f$ and ANOVA terms $f_{\u}$, respectively, and proved their properties as well as formulas for the Fourier coefficients. For functions in Sobolev type spaces $\st(\T^d)$ and the weighted Wiener algebra $\aw(\T^d)$ we showed that a function inherits its smoothness to both the projections and ANOVA terms.

Moreover, we related the smoothness of a function characterized by the decay of its Fourier coefficients to the class of functions of a low-dimensional structure and considered relative errors for $\mathrm{L}_\infty$ and $\mathrm{L}_2$ weighed by the corresponding Sobolev and Wiener algebra norms. This lead to an upper bound for the modified superposition dimension $\dsuper$ in those spaces. For product and order-dependent weights $w^{\alpha,\beta}$ we were able to obtain specific bounds.

We introduced an approximation method for high-dimensional functions that are of a low-dimensional structure in \Cref{sec:anova:approxmethod}. The method can be employed in black-box and scattered data approximation. In the former scenario one needs a special discretization for the index sets of type $I(U)$, e.g., rank-1 lattice, and in the latter an algorithm to realize an efficient multiplication with the Fourier matrices. We proved results for the error of the method, see \cref{sec:anova:error}, in $\Lt$ and $\mathrm{L}_\infty$. An $\L_\infty$ bound in the scattered data case for the aliasing error in \eqref{eq:approx_error} is still open. Here, one needs to consider estimating the quantity \begin{equation*}
\sup_{\norm{f}{\aw(\T^d)} \leq 1} \frac{1}{\abs{X}}\sum_{\x\in X} \abs{\left(f-S_{I(U)} f\right)(\x)}.
\end{equation*}

Numerical experiments with a benchmark function were successfully performed in \cref{sec:numeric}. The active set detection works well for this function in both approximation scenarios and even for small degrees of trigonometric polynomials. A definite goal is to perform experiments on real-world data sets and try to determine attribute rankings.

Moreover, it is possible to consider a similar analysis of the ANOVA decomposition in weighted Lebesgue spaces with orthogonal polynomials as bases, e.g., the Chebyshev system. This would also allow a generalization of the approximation method to a non-periodic setting, cf.\ \cite{PoSc19b}.

\section*{Acknowledgments}
First of all, we thank the referees for their valuable comments and suggestions. The authors also thank Lutz K\"ammerer, Toni Volkmer, and Tino Ullrich for fruitful discussions and remarks on the contents of the paper. Moreover, we thank Tino Ullrich for comments on \cref{thm:error:approx} and \cref{lem:error:l2est} which improved an earlier bound. DP acknowledges funding by Deutsche Forschungsgemeinschaft (German Research Foundation) -- Project--ID 416228727 -- SFB 1410. MS is supported by the BMBF grant 01$|$S20053A.

\bibliographystyle{siamplain} 
\bibliography{references.bib}

\begin{thebibliography}{10}

\bibitem{BaGn14}
{\sc J.~Baldeaux and M.~Gnewuch}, {\em Optimal randomized multilevel algorithms
  for infinite-dimensional integration on function spaces with {ANOVA}-type
  decomposition}, {SIAM} J. Numer. Anal., 52 (2014), pp.~1128--1155,
  \url{https://doi.org/10.1137/120896001}.

\bibitem{BaGr03}
{\sc R.~F. Bass and K.~Gr{{\"o}}chenig}, {\em Random sampling of multivariate
  trigonometric polynomials}, \textrm{SIAM} J. Math. Anal., 36 (2004),
  pp.~773--795, \url{https://doi.org/10.1137/S0036141003432316}.

\bibitem{Bj96}
{\sc {\AA}.~Bj{\"o}rck}, {\em Numerical Methods for Least Squares Problems},
  \textrm{SIAM}, Philadelphia, PA, USA, 1996.

\bibitem{BuGr04}
{\sc H.-J. Bungartz and M.~Griebel}, {\em Sparse grids}, Acta Numer., 13
  (2004), pp.~147--269, \url{https://doi.org/10.1017/S0962492904000182}.

\bibitem{ByKaUlVo16}
{\sc G.~{B}yrenheid, L.~{K}\"{a}mmerer, T.~{U}llrich, and T.~{V}olkmer}, {\em
  Tight error bounds for rank-1 lattice sampling in spaces of hybrid mixed
  smoothness}, Numer. Math., 136 (2017), pp.~993--1034,
  \url{https://doi.org/10.1007/s00211-016-0861-7}.

\bibitem{CaMoOw97}
{\sc R.~Caflisch, W.~Morokoff, and A.~Owen}, {\em Valuation of mortgage-backed
  securities using {B}rownian bridges to reduce effective dimension}, J.
  Comput. Finance, 1 (1997), pp.~27--46,
  \url{https://doi.org/10.21314/jcf.1997.005}.

\bibitem{DuTeUl16}
{\sc D.~{D{\~u}ng}, V.~N. {Temlyakov}, and T.~{Ullrich}}, {\em Hyperbolic Cross
  Approximation}, Advanced Courses in Mathematics -- CRM Barcelona,
  Birkh\"auser, Cham, 2018.

\bibitem{JoKuSl13}
{\sc J.~Dick, F.~Y. Kuo, and I.~H. Sloan}, {\em High-dimensional integration:
  The quasi-{M}onte {C}arlo way}, Acta Numer., 22 (2013), pp.~133--288,
  \url{https://doi.org/10.1017/S0962492913000044}.

\bibitem{DiSlWaWo06}
{\sc J.~Dick, I.~H. Sloan, X.~Wang, and H.~Wo{\'{z}}niakowski}, {\em Good
  lattice rules in weighted {K}orobov spaces with general weights}, Numer.
  Math., 103 (2006), pp.~63--97,
  \url{https://doi.org/10.1007/s00211-005-0674-6}.

\bibitem{DoKuPo09}
{\sc M.~D{\"o}hler, S.~Kunis, and D.~Potts}, {\em Nonequispaced hyperbolic
  cross fast {F}ourier transform}, {\textrm{SIAM}} J. Numer. Anal., 47 (2010),
  pp.~4415 -- 4428, \url{https://doi.org/10.1137/090754947}.

\bibitem{GiKuNuWa18}
{\sc A.~D. Gilbert, F.~Y. Kuo, D.~Nuyens, and G.~W. Wasilkowski}, {\em
  Efficient implementations of the multivariate decomposition method for
  approximating infinite-variate integrals}, {SIAM} J. Sci. Comput., 40 (2018),
  pp.~A3240--A3266, \url{https://doi.org/10.1137/17m1161890}.

\bibitem{GnHeHiRi17}
{\sc M.~Gnewuch, M.~Hefter, A.~Hinrichs, and K.~Ritter}, {\em Embeddings of
  weighted {H}ilbert spaces and applications to multivariate and
  infinite-dimensional integration}, J. Approx. Theory, 222 (2017), pp.~8--39,
  \url{https://doi.org/10.1016/j.jat.2017.05.003}.

\bibitem{GrKuNi14}
{\sc I.~G. Graham, F.~Y. Kuo, J.~A. Nichols, R.~Scheichl, C.~Schwab, and I.~H.
  Sloan}, {\em Quasi-{M}onte {C}arlo finite element methods for elliptic {PDEs}
  with lognormal random coefficients}, Numer. Math., 131 (2014), pp.~329--368,
  \url{https://doi.org/10.1007/s00211-014-0689-y}.

\bibitem{GrKuNu18}
{\sc I.~G. Graham, F.~Y. Kuo, D.~Nuyens, R.~Scheichl, and I.~H. Sloan}, {\em
  Circulant embedding with {QMC}: analysis for elliptic {PDE} with lognormal
  coefficients}, Numer. Math., 140 (2018), pp.~479--511,
  \url{https://doi.org/10.1007/s00211-018-0968-0}.

\bibitem{Gr05}
{\sc M.~Griebel}, {\em Sparse grids and related approximation schemes for
  higher dimensional problems}, in Foundations of Computational Mathematics
  (FoCM05), Santander, L.~Pardo, A.~Pinkus, E.~Suli, and M.~Todd, eds.,
  Cambridge University Press, 2006, pp.~106--161.

\bibitem{GriHa13}
{\sc M.~Griebel and J.~Hamaekers}, {\em Fast discrete {F}ourier transform on
  generalized sparse grids}, in Sparse Grids and Applications - Munich 2012,
  J.~Garcke and D.~Pfl\"uger, eds., vol.~97 of Lect. Notes Comput. Sci. Eng.,
  Springer International Publishing, 2014, pp.~75--107,
  \url{https://doi.org/10.1007/978-3-319-04537-5_4}.

\bibitem{GrHo10}
{\sc M.~Griebel and M.~Holtz}, {\em {D}imension-wise integration of
  high-dimensional functions with applications to finance}, J. Complexity, 26
  (2010), pp.~455--489, \url{https://doi.org/10.1016/j.jco.2010.06.001}.

\bibitem{GrKuSl10}
{\sc M.~Griebel, F.~Y. Kuo, and I.~H. Sloan}, {\em The smoothing effect of the
  {ANOVA} decomposition}, J. Complexity, 26 (2010), pp.~523--551,
  \url{https://doi.org/10.1016/j.jco.2010.04.003}.

\bibitem{GrKuSl16}
{\sc M.~Griebel, F.~Y. Kuo, and I.~H. Sloan}, {\em The {ANOVA} decomposition of
  a non-smooth function of infinitely many variables can have every term
  smooth}, Math. Comp., 86 (2016), pp.~1855--1876,
  \url{https://doi.org/10.1090/mcom/3171}.

\bibitem{GrPoRa07}
{\sc K.~Groechenig, B.~M. Poetscher, and H.~Rauhut}, {\em Learning
  trigonometric polynomials from random samples and exponential inequalities
  for eigenvalues of random matrices}, Tech. Report math.PR/0701781, 2007,
  \url{http://cds.cern.ch/record/1013574}.

\bibitem{He03}
{\sc M.~Hegland}, {\em Adaptive sparse grids}, in Proc. of 10th Computational
  Techniques and Applications Conference CTAC-2001, K.~Burrage and R.~B. Sidje,
  eds., vol.~44, 2003, pp.~C335--C353,
  \url{https://doi.org/10.21914/anziamj.v44i0.685}.

\bibitem{HeLe11}
{\sc M.~Hegland and P.~Leopardi}, {\em The rate of convergence of sparse grid
  quadrature on the torus}, ANZIAM Journal, 52 (2011),
  \url{https://doi.org/10.21914/anziamj.v52i0.3952}.

\bibitem{Holtz11}
{\sc M.~Holtz}, {\em Sparse grid quadrature in high dimensions with
  applications in finance and insurance}, vol.~77 of Lecture Notes in
  Computational Science and Engineering, Springer-Verlag, Berlin, 2011,
  \url{https://doi.org/10.1007/978-3-642-16004-2}.

\bibitem{Kae2012}
{\sc L.~K{\"a}mmerer}, {\em Reconstructing hyperbolic cross trigonometric
  polynomials by sampling along rank-1 lattices}, SIAM J. Numer. Anal., 51
  (2013), pp.~2773--2796, \url{http://dx.doi.org/10.1137/120871183}.

\bibitem{Kae2013}
{\sc L.~K{\"a}mmerer}, {\em Reconstructing multivariate trigonometric
  polynomials from samples along rank-1 lattices}, in Approximation Theory XIV:
  San Antonio 2013, G.~E. Fasshauer and L.~L. Schumaker, eds., Springer
  International Publishing, 2014, pp.~255--271,
  \url{https://doi.org/10.1007/978-3-319-06404-8_14}.

\bibitem{KaPoVo13}
{\sc L.~K\"ammerer, D.~Potts, and T.~Volkmer}, {\em Approximation of
  multivariate periodic functions by trigonometric polynomials based on rank-1
  lattice sampling}, J. Complexity, 31 (2015), pp.~543--576,
  \url{https://doi.org/10.1016/j.jco.2015.02.004}.

\bibitem{KaPoVo14}
{\sc L.~K\"ammerer, D.~Potts, and T.~Volkmer}, {\em {Approximation of
  multivariate periodic functions by trigonometric polynomials based on
  sampling along rank-1 lattice with generating vector of Korobov form}}, J.
  Complexity, 31 (2015), pp.~424--456,
  \url{https://doi.org/10.1016/j.jco.2014.09.001}.

\bibitem{KaPoVo17}
{\sc L.~K\"{a}mmerer, D.~Potts, and T.~Volkmer}, {\em High-dimensional sparse
  {FFT} based on sampling along multiple rank-1 lattices}, Appl. Comput.
  Harmon. Anal., 51 (2021), pp.~225--257,
  \url{https://doi.org/10.1016/j.acha.2020.11.002}.

\bibitem{KaeUlVo19}
{\sc L.~{K{\"a}mmerer}, T.~{Ullrich}, and T.~{Volkmer}}, {\em {Worst-case
  recovery guarantees for least squares approximation using random samples}},
  ArXiv e-prints 1911.10111,  (2019).

\bibitem{KaVo19}
{\sc L.~K{\"a}mmerer and T.~Volkmer}, {\em Approximation of multivariate
  periodic functions based on sampling along multiple rank-1 lattices}, J.
  Approx. Theory, 246 (2019), pp.~1--27,
  \url{https://doi.org/10.1016/j.jat.2019.05.001}.

\bibitem{KeKuPo09}
{\sc J.~Keiner, S.~Kunis, and D.~Potts}, {\em Using {NFFT3} - a software
  library for various nonequispaced fast {Fourier} transforms}, {ACM} Trans.
  Math. Software, 36 (2009), pp.~Article 19, 1--30,
  \url{https://doi.org/10.1145/1555386.1555388}.

\bibitem{KuMaUl16}
{\sc T.~K\"{u}hn, S.~Mayer, and T.~Ullrich}, {\em Counting via entropy: New
  preasymptotics for the approximation numbers of {S}obolev embeddings}, {SIAM}
  J. Numer. Anal., 54 (2016), pp.~3625--3647,
  \url{https://doi.org/10.1137/16m106580x}.

\bibitem{KuNu16}
{\sc F.~Y. Kuo and D.~Nuyens}, {\em Application of quasi-{M}onte {C}arlo
  methods to elliptic {PDEs} with random diffusion coefficients: A survey of
  analysis and implementation}, Found. Comput. Math, 16 (2016), pp.~1631--1696,
  \url{https://doi.org/10.1007/s10208-016-9329-5}.

\bibitem{KuNuPlSlWa17}
{\sc F.~Y. Kuo, D.~Nuyens, L.~Plaskota, I.~H. Sloan, and G.~W. Wasilkowski},
  {\em {I}nfinite-dimensional integration and the multivariate decomposition
  method}, J. Comput. Appl. Math., 326 (2017), pp.~217--234,
  \url{https://doi.org/10.1016/j.cam.2017.05.031}.

\bibitem{KuSchwSl12}
{\sc F.~Y. Kuo, C.~Schwab, and I.~H. Sloan}, {\em Quasi-{M}onte {C}arlo finite
  element methods for a class of elliptic partial differential equations with
  random coefficients}, SIAM J. Numer. Anal., 50 (2012), pp.~3351 -- 3374,
  \url{https://doi.org/10.1137/110845537}.

\bibitem{KuSlWaWo09}
{\sc F.~Y. Kuo, I.~H. Sloan, G.~W. Wasilkowski, and H.~Wo{\'{z}}niakowski},
  {\em On decompositions of multivariate functions}, Math. Comp., 79 (2010),
  pp.~953--966, \url{https://doi.org/10.1090/s0025-5718-09-02319-9}.

\bibitem{LiHi03}
{\sc D.~Li and F.~J. Hickernell}, {\em Trigonometric spectral collocation
  methods on lattices}, in Recent Advances in Scientific Computing and Partial
  Differential Equations, S.~Y. Cheng, C.-W. Shu, and T.~Tang, eds., vol.~330
  of Contemp. Math., AMS, 2003, pp.~121--132,
  \url{https://doi.org/10.1090/conm/330/05887}.

\bibitem{LiOw06}
{\sc R.~Liu and A.~B. Owen}, {\em Estimating mean dimensionality of analysis of
  variance decompositions}, J. Amer. Statist. Assoc., 101 (2006), pp.~712--721,
  \url{https://doi.org/10.1198/016214505000001410}.

\bibitem{MoeUl20}
{\sc M.~{Moeller} and T.~{Ullrich}}, {\em {$L_2$-norm sampling discretization
  and recovery of functions from RKHS with finite trace}}, ArXiv e-prints
  2009.11940,  (2020).

\bibitem{Ni92}
{\sc H.~Niederreiter}, {\em {Random Number Generation and Quasi-{M}onte {C}arlo
  Methods}}, CBMS-NSF Regional Conference Series in Applied Mathematics,
  Society for Industrial and Applied Mathematics, 1992,
  \url{https://doi.org/10.1137/1.9781611970081}.

\bibitem{Owen2019}
{\sc A.~Owen}, {\em Effective dimension of some weighted pre-{S}obolev spaces
  with dominating mixed partial derivatives}, {SIAM} J. Numer. Anal., 57
  (2019), pp.~547--562, \url{https://doi.org/10.1137/17m1158975}.

\bibitem{PaSa82}
{\sc C.~C. Paige and M.~A. Saunders}, {\em {LSQR}: {A}n algorithm for sparse
  linear equations and sparse least squares}, ACM Trans. Math. Software, 8
  (1982), pp.~43--71, \url{https://doi.org/10.1145/355984.355989}.

\bibitem{PlPoStTa18}
{\sc G.~Plonka, D.~Potts, G.~Steidl, and M.~Tasche}, {\em Numerical Fourier
  Analysis}, Applied and Numerical Harmonic Analysis, Birkh\"auser, 2018,
  \url{https://doi.org/10.1007/978-3-030-04306-3}.

\bibitem{PoSc19b}
{\sc D.~Potts and M.~Schmischke}, {\em Learning multivariate functions with
  low-dimensional structures using polynomial bases}, ArXiv e-prints
  1912.03195,  (2019).

\bibitem{RaAl99}
{\sc H.~Rabitz and O.~F.~Alis}, {\em General foundations of high dimensional
  model representations}, J. Math. Chem., 25 (1999), pp.~197--233,
  \url{https://doi.org/10.1023/A:1019188517934}.

\bibitem{SlWo01}
{\sc I.~H. Sloan and H.~Wo\'zniakowski}, {\em Tractability of multivariate
  integration for weighted {K}orobov classes}, J. Complexity, 17 (2001),
  pp.~697--721, \url{https://doi.org/10.1006/jcom.2001.0599}.

\bibitem{So90}
{\sc I.~M. Sobol}, {\em {O}n sensitivity estimation for nonlinear mathematical
  models}, Keldysh AppliedMathematics Institute, 1 (1990), pp.~112--118.

\bibitem{So01}
{\sc I.~M. Sobol}, {\em Global sensitivity indices for nonlinear mathematical
  models and their {M}onte {C}arlo estimates}, Math. Comput. Simulation, 55
  (2001), pp.~271--280, \url{https://doi.org/10.1016/s0378-4754(00)00270-6}.

\end{thebibliography}
\end{document}